\documentclass[10pt]{amsart}

\usepackage{amsfonts,latexsym,amsmath,amssymb,amsthm,amscd,upgreek}
\usepackage{mathrsfs,manfnt,enumerate}
\usepackage[super]{nth}
\usepackage{setspace}

\usepackage{verbatim,enumitem}
\usepackage{calc}
\usepackage[margin=1in]{geometry}
\usepackage{amsxtra}
\usepackage[all]{xy}
\usepackage{graphicx}
\usepackage{hyperref}
\hypersetup{colorlinks=true,citecolor=blue,urlcolor =black,linkbordercolor={1 0 0}}
\usepackage{mathtools}
\usepackage{pifont}
\usepackage{tikz-cd}
\usetikzlibrary{matrix,shapes,arrows,decorations.pathmorphing,calc}
\usepackage{bm}
\usetikzlibrary{decorations.pathreplacing,decorations.markings}
\usepackage[new]{old-arrows}
\usepackage{calculator}
\usepackage{scalerel,xcolor}
\def\darrow{\mathrel{\ThisStyle{\ooalign{$\SavedStyle\rightarrow$\cr%
				\hfil\textcolor{white}{\rule{2\LMpt}{1\LMex}}\kern2\LMpt\hfil}}}}

\newtheorem{theorem}{Theorem}

\newtheorem{proposition}[theorem]{Proposition}
\newtheorem{lemma}[theorem]{Lemma}
\newtheorem{corollary}[theorem]{Corollary}

\theoremstyle{definition}

\newtheorem{example}[theorem]{Example}
\newtheorem{prop}[theorem]{Proposition}

\newtheorem{Definition}[theorem]{Definition}
\newenvironment{definition}
{\begin{Definition}\rm}{\end{Definition}}

\theoremstyle{remark}
\newtheorem{remark}[theorem]{Remark}

\newcommand{\sign}{\operatorname{sign}}
\newcommand{\id}{\operatorname{id}}
\newcommand{\defeq}{\vcentcolon=}

\newcommand\nc{\newcommand}
\nc{\on}{\operatorname}
\nc\renc{\renewcommand}
\nc{\BR}{\mathbb R}
\nc{\BC}{\mathbb C}
\nc{\BQ}{\mathbb Q}
\nc{\BF}{\mathbb F}
\nc{\BZ}{\mathbb Z}
\nc{\BN}{\mathbb N}
\nc{\BS}{\mathbb S}
\nc{\BA}{\mathbb A}
\nc{\BP}{\mathbb P}
\nc{\Hom}{\on{Hom}}
\nc{\wt}{\widetilde}
\nc{\vspan}{\on{span}}
\nc{\ord}{\on{ord}}
\nc{\im}{\on{im}}
\nc{\Mat}{\on{Mat}}
\nc{\can}{\on{can}}
\nc{\coker}{\on{coker}}
\nc{\ev}{\on{ev}}
\nc{\Tr}{\on{Tr}}
\nc{\End}{\on{End}}
\nc{\Aut}{\on{Aut}}
\nc{\swap}{\on{swap}}
\nc{\Set}{\on{Set}}
\nc{\bC}{{\mathbf C}}
\nc{\bc}{{\mathbf c}}
\nc{\bD}{{\mathbf D}}
\nc{\bd}{{\mathbf d}}
\nc{\bE}{{\mathbf E}}
\nc{\be}{{\mathbf e}}
\nc{\bF}{{\mathbf F}}
\nc{\bff}{{\mathbf f}}
\nc{\fa}{\mathfrak a}

\nc{\adj}{\on{adj}}
\nc{\tensor}[3]{#1 \underset{#2}\otimes #3}
\nc{\Nat}{\on{Nat}}
\nc{\op}{\on{op}}
\nc{\Funct}{\on{Funct}}
\nc{\Ob}{\on{Ob}}
\nc{\fR}{\mathfrak{R}}
\nc{\Vect}{\on{Vect}}
\nc{\ns}{\on{non-spec}}
\nc{\GL}{\on{GL}}
\nc{\ol}{\overline}
\nc{\ul}{\underline}
\nc{\univ}{\on{univ}}
\nc{\Maps}{\on{Maps}}
\nc{\bdd}{\on{bdd}}
\nc{\cont}{\on{cont}}
\nc{\Sym}{\on{Sym}}
\nc{\Ind}{\on{Ind}}
\nc{\Res}{\on{Res}}
\nc{\Ann}{\on{Ann}}
\nc{\cI}{\mathcal{I}}
\nc{\pt}{\on{pt}}
\nc{\Bl}{\on{\Bl}}
\nc{\Spec}{\on{Spec}}
\nc{\Cl}{\on{Cl}}
\nc{\cD}{\mathcal{D}}
\renc{\div}{\on{div}}
\nc{\mc}{\mathcal}
\nc{\pp}{\mathfrak{p}}
\nc{\scr}{\mathscr}
\nc{\purple}{\textcolor[rgb]{0,0,0}}
\nc{\Disc}{\on{Disc}}
\nc{\wh}{\widehat}

\newcommand{\inv}{^{-1}}

\newcommand{\bmat}[1]{\begin{bmatrix}#1\end{bmatrix}}

\newcommand{\vphi}{\varphi}

\newcommand{\disc}{\mathrm{disc}}

\newcommand{\red}{\mathrm{red}}

\newcommand{\St}{\mathrm{St}}

\newcommand{\proj}{\mathrm{proj}}

\newcommand{\rank}{\mathrm{rank}}

\newcommand{\A}{\mathbb{A}}
\newcommand{\C}{\mathbb{C}}

\newcommand{\Q}{\mathbb{Q}}
\newcommand{\R}{\mathbb{R}}

\newcommand{\Z}{\mathbb{Z}}

\newcommand{\calh}{\mathcal{H}}
\newcommand{\cali}{\mathcal{I}}

\newcommand{\call}{\mathcal{L}}

\newcommand{\calp}{\mathcal{P}}

\newcommand{\cals}{\mathcal{S}}

\newcommand{\calv}{\mathcal{V}}

\newcommand{\fb}{\mathfrak{b}}

\newcommand{\fp}{\mathfrak{p}}

\newcommand{\fP}{\mathfrak{P}}

\DeclareMathOperator{\Stab}{\mathrm{Stab}}
\DeclareMathOperator\exterior{\rotatebox[origin=c]{180}{\textsf{V}}}

\newcommand{\SL}{\mathrm{SL}}

\makeatletter
\newcommand{\extp}{\@ifnextchar^\@extp{\@extp^{\,}}}
\def\@extp^#1{\mathop{\bigwedge\nolimits^{\!#1}}}
\makeatother

\tikzset{%
	add/.style args={#1 and #2}{
		to path={%
			($(\tikztostart)!-#1!(\tikztotarget)$)--($(\tikztotarget)!-#2!(\tikztostart)$)%
			\tikztonodes},add/.default={.2 and .2}}
}

\tikzset{
	on each segment/.style={
		decorate,
		decoration={
			show path construction,
			moveto code={},
			lineto code={
				\path [#1]
				(\tikzinputsegmentfirst) -- (\tikzinputsegmentlast);
			},
			curveto code={
				\path [#1] (\tikzinputsegmentfirst)
				.. controls
				(\tikzinputsegmentsupporta) and (\tikzinputsegmentsupportb)
				..
				(\tikzinputsegmentlast);
			},
			closepath code={
				\path [#1]
				(\tikzinputsegmentfirst) -- (\tikzinputsegmentlast);
			},
		},
	},
	rightend arrow/.style={postaction={decorate,decoration={
				markings,
				mark=at position .81 with {\arrow[#1]{triangle 45}}
	}}},
	leftend arrow/.style={postaction={decorate,decoration={
				markings,
				mark=at position .18 with {\arrowreversed[#1]{triangle 45}}
	}}},
}

\makeatletter
\@namedef{subjclassname@2020}{%
	\textup{2020} Mathematics Subject Classification}
\makeatother

\makeatletter
\def\@tocline#1#2#3#4#5#6#7{\relax
	\ifnum #1>\c@tocdepth 
	\else
	\par \addpenalty\@secpenalty\addvspace{#2}%
	\begingroup \hyphenpenalty\@M
	\@ifempty{#4}{%
		\@tempdima\csname r@tocindent\number#1\endcsname\relax
	}{%
		\@tempdima#4\relax
	}%
	\parindent\z@ \leftskip#3\relax \advance\leftskip\@tempdima\relax
	\rightskip\@pnumwidth plus4em \parfillskip-\@pnumwidth
	#5\leavevmode\hskip-\@tempdima
	\ifcase #1
	\or\or \hskip 30pt \or \hskip 2em \else \hskip 3em \fi%
	#6\nobreak\relax
	\hfill\hbox to\@pnumwidth{\@tocpagenum{#7}}\par
	\nobreak
	\endgroup
	\fi}
\makeatother

\usepackage[draft]{todonotes}

\allowdisplaybreaks

\title[A parametrization of $3$-class groups of quadratic rings over Dedekind domains]{\vspace*{-0.6in} A parametrization of $3$-class groups of quadratic rings \\ over Dedekind domains}

\author{\vspace*{-1pt}Eliot Hodges and Ashvin A.~Swaminathan}

\AtEndDocument{\medskip{\footnotesize%
		\noindent\textsc{Department of Mathematics, Harvard University, \mbox{Camrbidge, MA 02138}} 
		\par \textit{E-mail address}, Eliot Hodges: \texttt{eliotljhodges@gmail.com}
		\par \textit{E-mail address}, Ashvin A. Swaminathan: \texttt{ashvins@alumni.princeton.edu}
		
}}

\begin{document}
	
	\maketitle
	
	\vspace*{-0.3in}
	\vspace*{-1pt}
	\begin{abstract}
		Let $R$ be a Dedekind domain with field of fractions $K$ and $\on{char}(R)\neq3$. In this paper, we generalize Bhargava's parametrization of 3-torsion ideal classes by binary cubic forms to work over $R$. Specifically, we construct arithmetic subgroups of $\on{GL}_2(K)$ whose actions on certain lattices of binary cubic forms over $K$ parametrize 3-torsion ideal classes in class groups of quadratic rings over $R$. 
	\end{abstract}
	
	\tableofcontents
	
	\vspace*{-0.2in}
	\section{Introduction} \label{sec-intro}
	
	In his doctoral thesis, Bhargava discovered fourteen generalizations of Gauss's composition law for binary quadratic forms \cite{BhargavaHCLI,HCLII,HCLIII,HCLIV}. These ``higher composition laws'' yielded beautiful orbit parametrizations for a plethora of interesting arithmetic objects, from rings of small rank to ideal classes of such rings. In turn, these parametrizations served as the core algebraic ingredients for an extensive and ever-growing series of results in arithmetic statistics concerning asymptotic counts for number fields and their class groups; see, e.g., \cite{MR2183288,MR2745272,MR3090184,MR3369305,BV,BSHpreprint,MR4891959}.

	Over the last two decades, there has been significant effort dedicated to generalizing Bhargava's higher composition laws to work not just over $\Z$ but over an arbitrary base. Of particular salience are Wood's doctoral thesis \cite{WoodGaussComp,Wood11.2,Wood11,Wood11.3}, in which she proved versions of many of Bhargava's results over an arbitrary base scheme, and O'Dorney's undergraduate thesis \cite{ODorney}, in which he gave explicit reinterpretations of some of Wood's results over an arbitrary Dedekind domain. However, this type of generalization has not been formulated previously for one of Bhargava's most fundamental parametrizations, namely that of $3$-class groups of quadratic rings by $\SL_2(\Z)$-orbits of integral binary cubic forms with triplicate central coefficients. In this paper, we generalize this parametrization to work over an arbitrary Dedekind domain $R$ with $\on{char}(R)\neq3$.
	
	Na\"ively extending the parametrization of $3$-class groups to work over such base rings $R$ fails on account of the Steinitz class, which is a natural invariant attached to finite-dimensional $R$-lattices that measures the extent to which such lattices fail to be free. In particular, Bhargava's approach to constructing a ring and $3$-torsion ideal class from a binary cubic form over $R$ only yields quadratic rings that have trivial Steinitz class---i.e., are free over $R$---to which the parametrization was already generalized by Bhargava, Elkies, and Shnidman, under the additional assumption that the characteristic of $R$ is not $2$ or $3$ \cite{BES}. It is natural to seek to lift these restrictions---e.g., by a result of Kable and Wright \cite{KW}, the Steinitz classes of rings of integers of quadratic extensions of number fields are equidistributed, so the Bhargava--Elkies--Shnidman result captures only a fraction of such quadratic extensions. The purpose of this paper is to generalize the parametrization completely, without restrictions on the Steinitz class.
	
	\subsection{Main result} 
	Let $R$ be a Dedekind domain with $\on{char}(R)\neq3$ and fraction field $K$, and fix a fractional ideal $\fa$ of $R$. Let $S$ be a quadratic ring over $R$ with Steinitz class $\on{St}(S) = [\fa]$, and let $L \defeq S \otimes_R K$. Then $S$ can be written as $S=R+\fa\xi$ for some $\xi \in L^{\times}$. The decomposition $\Cl(S)[3]\simeq\Cl(R)[3]\oplus\Cl(S/R)[3]$ (see Proposition~\ref{prop-classsplits}) implies that it suffices to parametrize the \emph{relative} 3-class group $\on{Cl}(S/R)[3]$, which is the group of $3$-torsion ideal classes $I$ of $S$ with principal ideal norm.
	
	Recall that the group $\GL_2(K)$ acts on  binary cubic forms over $K$ via $g\cdot f(x,y)=\det(g)\inv f((x,y)g)$. We will parametrize relative $3$-torsion classes using orbits for the action of certain arithmetic subgroups of $\on{GL}_2(K)$ on certain lattices of binary cubic forms over $K$. Indeed, consider the subgroup
	\[\GL(R\oplus\fa)\defeq\Aut(R\oplus\fa)=\left\{\bmat{a&b\\c&d}\mid a,d\in R,\,b\in\fa,\,c\in\fa\inv,\,ad-bc\in R^\times\right\} \subset \on{GL}_2(K).\] 
	The action of $\GL_2(K)$ on binary cubic forms restricts to a natural $\GL(R\oplus\fa)$-action on the subspace \[V_\fa\defeq\{ax^3+3bx^2y+3cxy^2+dy^3\mid a\in\fa,\,b\in R,\,c\in\fa\inv,\,d\in\fa^{-2}\}.\footnote{Here we see the need for the additional stipulation that $\on{char}(R)\neq3$. Over a ring with characteristic 3, it is possible for inequivalent quadruples $(S,I,\delta,s)$ and $(S,I',\delta',s')$ to map to the same element of $V_\fa$ under the bijection described in Theorem~\ref{thm:mainresult}. In characteristic 3, equivalence classes of quadruples satisfying the hypotheses of Theorem~\ref{thm:mainresult} can instead be parametrized by \emph{triply symmetric Bhargava cubes.} See \cite{ODorney} for a definition of Bhargava cubes over a Dedekind domain.}\]  For $f\in V_\fa$, the discriminant $\disc(f)\in\fa^{-2}$ of $f$ is defined to be $-1/27$ times the usual polynomial discriminant of $f$ and is invariant under the action of $\on{GL}(R \oplus \fa)$. We are now in position to state our generalization Bhargava's parametrization to Dedekind domains:

	\begin{theorem}[cf.\ Theorem~\ref{thm: GL parametrization}] \label{thm:mainresult}
		With notation as above, there is a canonical bijection between the set of nondegenerate $\GL(R\oplus\fa)$-orbits on $V_\fa$ and equivalence classes of quadruples $(S,I,\delta,s)$, where $S$ is a nondegenerate quadratic ring, $I$ is a fractional ideal of $S$ with $\St(I)=[\fa]=\St(S)$, $\delta \in (S\otimes_RK)^\times$, and $s\in K^\times$ are such that $I^3\subset\delta S$, the module index of $I$ in $S$ is $sR$, and $s^3=N(\delta)$. Here, we take $(S,I,\delta,s) \approx(S',I',\delta',s')$ if there exists a ring isomorphism $\phi\colon S\to S'$ and $\kappa\in S'\otimes_RK$ such that $I'=\kappa \phi(I)$, $\delta'=\kappa^3\phi(\delta)$, and $s'=N(\kappa)\phi(s)$. Under this bijection, the discriminant of a binary cubic form multiplied by $\fa^2$ is the discriminant ideal of the corresponding quadratic ring.
	\end{theorem}

	A binary cubic form $f(x,y)=ax^3+3bx^2y+3cxy^2+dy^3\in V_\fa$ is called \emph{projective} if its Hessian covariant \[H(x,y)=(b^2-ac)x^2+(ad-bc)xy+(c^2-bd)y^2\] is a \emph{primitive} binary quadratic form (see \cite[\S2.1]{WoodGaussComp} or Definition~\ref{def: BQF}). That is, we say that $f$ is projective if the ideals $(b^2-ac)R$, $(ad-bc)\fa$, and $(c^2-bd)\fa^2$ do not share a common prime factor. Under the correspondence described in Theorem~\ref{thm:mainresult}, the $\GL(R\oplus\fa)$-orbits of projective forms in $V_\fa$ correspond precisely to the quadruples $(S,I,\delta,s)$ such that $I$ is invertible. Denote the subset of projective forms in $V_\fa$ by $V_\fa^\proj$.
	
	When $R$ has characteristic not $2$, the isomorphism classes of quadratic rings over $R$ are determined uniquely by their Steinitz classes and their discriminants. For $\Delta\in\fa^{-2}$ and $W\subset V_\fa$, let $W^{\Delta} \defeq \{f \in W \mid \disc(f)=\Delta\}$. Then we have the following arithmetic consequence of Theorem~\ref{thm:mainresult}.
	\begin{corollary}[cf.\ Theorem~\ref{thm: form ideal correspondence}]\label{cor: orbit corr}
		Let $R$ be the ring of integers of a number field $K$, and let $S$ be an order in a quadratic extension $L$ of $K$. Suppose that $\fa$ is a representative of $\on{St}(S)$ so that $S$ has discriminant ideal given by $\Delta\fa^2$ for some $\Delta\in\fa^{-2}$. Then we have that
		\[\#\Cl(S/R)[3]=\frac{1}{3^{s+\varepsilon}}\#(\GL(R\oplus\fa)\backslash V_\fa^{\proj,\Delta}),\] where $s$ is the number of real places of $K$ that split in $L$ and where $\varepsilon=1$ if $S$ contains the cube roots of unity and $R$ does not and $\varepsilon=0$ otherwise.
	\end{corollary}
	
	In forthcoming work \cite{HSAnalysis}, we use Corollary~\ref{cor: orbit corr} to compute the average number of 3-torsion ideal classes in orders of quadratic extension of number fields \cite{BV}. 
	
	We now describe the bijection from Theorem~\ref{thm:mainresult} in greater detail. While the result is stated in terms of $\GL(R\oplus\fa)$-orbits of $V_\fa$, its proof involves first classifying the $\SL(R\oplus\fa)$-orbits of $V_\fa$ and then passing to the larger group. The $\SL(R\oplus\fa)$-orbits of $V_\fa$ correspond to equivalence classes of quadruples $(S,I,\delta,s)$ whose ring $S$ is \emph{oriented}. An $\fa$-\emph{orientation} of a quadratic ring $S$ over $R$ with Steinitz class $[\fa]$ is an isomorphism $\pi\colon \exterior^2(S)\to\fa$, and the orientation data is equivalent to a choice of basis $1,\xi$ of $S$ such that $S=R+\fa\xi$.
	
	Consider a quadruple $(S,I,\delta,s)$ satisfying the hypotheses of Theorem~\ref{thm:mainresult}. Suppose that $S$ is equipped with an orientation $\pi\colon \exterior^2(S)\to\fa$, and suppose further that $\pi$ yields the decomposition $S=R+\fa\xi$ with $\pi_K(1\wedge\xi)=1$, where $\pi_K$ denotes the (unique) extension of $\pi$ to an isomorphism $\exterior^2(S\otimes_RK)\to K$. Given the data of $(S,I,\delta,s)$ and $\pi$, we illustrate how to construct an binary cubic form in $V_\fa$. Begin by choosing a basis $I=R\alpha+\fa\beta$ such that $\pi_K(\alpha\wedge\beta)=s$. Since $I^3\subset\delta S$, we have that \begin{equation}
		\begin{split}
			\alpha^3&=\delta(c_0+a_0\xi)\\
			\alpha^2\beta&=\delta(c_1+a_1\xi)\\
			\alpha\beta^2&=\delta(c_2+a_2\xi)\\
			\beta^3&=\delta(c_3+a_3\xi)
		\end{split}
	\end{equation} for $a_0\in \fa$, $a_1\in R$, $a_2\in\fa\inv$, and $a_3\in\fa^{-2}$. The form corresponding to $(S,I,\delta,s)$ is given by \[C(x,y)=a_0x^3+3a_1x^2y+3a_2xy^2+a_3y^3.\] We can use the orientation isomorphism to elicit a basis-independent construction: noticing that \[C(x,y)=\pi\left(1\wedge\frac{(\alpha x+\beta y)^3}{\delta}\right)\] allows us to describe $C$ equivariantly as the map $I\to\fa$ taking $\zeta\mapsto\pi(1\wedge\zeta^3\delta\inv)$. The content of the proof of Theorem~\ref{thm:mainresult} is the construction an inverse map from $\SL(R\oplus\fa)\backslash V_\fa$ to the set of equivalence classes of quadruples satisfying the hypotheses of Theorem~\ref{thm:mainresult}. 
	
	\subsection{Arithmetic implications}
	While Bhargava's 3-class group parametrization is of independent interest, its arithmetic consequences---first laid out by Bhargava and Varma \cite{BV}---are manifold. Specifically, Bhargava and Varma used this parametrization to obtain, without appealing to class field theory, a direct proof of Davenport and Heilbronn's famous computation of the average size of the 3-class groups of quadratic fields \cite{DH}. Moreover, because Bhargava's theorem parametrizes 3-torsion ideals in all quadratic rings over $\Z$ (including both maximal and nonmaximal orders in quadratic fields), Bhargava and Varma were able to obtain an analogue of the Davenport--Heilbronn average for all quadratic orders over $\Z$ \cite{BV}.
	
	Theorem~\ref{thm:mainresult} makes possible the generalization of Bhargava and Varma's results to an arbitrary number field. In forthcoming work \cite{HSAnalysis}, we apply Theorem~\ref{thm:mainresult} and geometry-of-numbers techniques due to Bhargava--Shankar--Wang \cite{BSW} and Shankar--Siad--Swaminathan--Varma \cite{SSSV} to find the average number of 3-torsion elements in class groups of orders in quadratic extensions of an arbitrary number field. Likewise, it is possible to apply similar techniques to elicit a direct proof (without making use of Shintani zeta functions or class field theory) of Datskovsky and Wright's computation of the average number of 3-torsion elements in class groups of quadratic extensions of number fields \cite{DW88}. Datskovsky and Wright's approach makes use of Shintani zeta functions; this result was later reproved by Bhargava, Shankar, and Wang using class field theory and an asymptotic for the number of cubic extensions of a number field of bounded discriminant \cite{BSW}.

	While many parametrizations of arithmetic objects over Dedekind domains (and more general bases) exist in principle, few have been effectively applied to obtain asymptotics over arbitrary global fields. Bhargava, Shankar, and Wang circumvent the complications introduced by the Steinitz class by constructing orbits over fields, where the Steinitz class is trivial, and then applying an ad\`{e}lic trick to extend the parametrization to Dedekind domains. This approach is well-suited to their setting: they focus on 3-class groups of maximal orders, where rational orbits cannot contain multiple integral orbits. Thus, it suffices in their context to parametrize rational orbits and subsequently establish the existence of integral representatives. In contrast, for nonmaximal orders, rational orbits parametrizing ideal classes can have multiple inequivalent integral representatives. Theorem~\ref{thm:mainresult} provides an explicit Steinitz-indexed parametrization over Dedekind domains that directly accounts for all integral orbits. In \S\ref{sec: reduction theory}, we connect the two perspectives by presenting an ad\`{e}lic reformulation of the arithmetic groups $\GL(R \oplus \fa)$ in a form compatible with the construction of Bhargava, Shankar, and Wang.

	The rest of this paper is organized as follows. In \S\ref{subsec: alg over a DD}, we develop the background necessary for stating and proving Theorem~\ref{thm:mainresult}. We then give an intermediate parametrization of quadratic rings over $R$ in \S\ref{sec: param of quad orders and gauss comp} and conclude the section with a discussion of Gauss composition over a Dedekind domain. The entirety of \S\ref{subsec: param} is dedicated to proving Theorem~\ref{thm:mainresult} and several of its corollaries. Finally, in \S\ref{sec: reduction theory}, we study the $\GL_2$-action on triply symmetric binary cubic forms over certain local fields and offer an ad\`elic reinterpretation of the group $\GL(R\oplus\fa)$ when $R$ is the ring of integers of a global field, facilitating its use in orbit-counting scenarios.

	\section{Background on algebra over a Dedekind domain}\label{subsec: alg over a DD}

	In this section, we build the algebraic background material needed for the proofs of our main results in \S\S\ref{sec: param of quad orders and gauss comp}--\ref{subsec: param}. We first recall the basic theory of modules and algebras over such domains. Then we define the module index, discriminants, and algebra orientations and consider the interplay between some of these objects. We conclude the section with a discussion of class groups of algebras over Dedekind domains and their torsion subgroups.
	
	Throughout this section, $R$ denotes a Dedekind domain with \mbox{field of fractions $K$.} For simplicity, the modules and algebras we deal with will always have finite rank over $R$. Given a prime ideal $\fp$ of $R$, let $R_{(\fp)}$ denote the localization of $R$ away from $\fp$. 
	
	\subsection{Lattices over $R$} We start by defining the notion of a lattice over $R$:
	
	\begin{definition}\label{def: lattice}
		A \emph{lattice} $M$ over $R$ is a finitely generated torsion-free $R$-module. The \emph{rank} of the lattice $M$ is the dimension of $M\otimes_RK$ as a $K$-vector space. 
	\end{definition}
	The following result of Steinitz gives a complete classification of lattices over $R$ up to isomorphism:
	\begin{theorem}[Steinitz; e.g., {\cite[Theorem~1.32]{Narkiewicz}}]\label{thm: Steinitz}
		Every lattice $M$ over $R$ is isomorphic to a direct sum of nonzero fractional ideals of $R$: \[M\simeq\fa_1\oplus\cdots\oplus\fa_m\simeq R^{m-1}\oplus\fa_1\cdots\fa_m.\] Two such direct sums $\fa_1\oplus\cdots\oplus\fa_m$ and $\fb_1\oplus\cdots\oplus\fb_n$ are isomorphic as $R$-modules if and only if $m=n$ and $\fa_1\cdots\fa_m$ and $\fb_1\cdots\fb_n$ belong to the same ideal class $($cf.~Definition~\ref{def: class group}$)$.
	\end{theorem}
	
	The number $m$ in Theorem~\ref{thm: Steinitz} is called the \emph{rank} of $M$ and is denoted $\on{rk}(M)$; the ideal class represented by $\fa_1\cdots\fa_m$ is called its \emph{Steinitz class} and is denoted $\St(M)$. The decomposition of $M$ into a direct sum of ideals of $R$ is called a \emph{Steinitz decomposition}, and the $R$-module isomorphism class of $M$ is determined entirely by $\on{rk}(M)$ and $\St(M)$. In particular, we have the following consequence of Theorem~\ref{thm: Steinitz}:
	\begin{corollary} \label{cor:steinitz}
		Let $M$ be a lattice of rank $n$ over $R$ with $\St(M)=[\fa]$ for some fractional ideal $\fa$ of $R$. For any fractional ideals $\fa_1,\ldots,\fa_n$ of $R$ such that $[\fa]=[\fa_1\cdots\fa_n]$, there are isomorphisms \[M\simeq \fa_1\oplus\cdots\oplus\fa_n \simeq R^{n-1}\oplus\fa_1\cdots\fa_n\simeq R^{n-1}\oplus\fa.\]
	\end{corollary}
	
	\begin{remark}
		When $R$ is a principal ideal domain (PID), the classification presented in Theorem~\ref{thm: Steinitz} reduces to the usual classification of finitely generated torsion-free modules over a PID: indeed, the ideal class group of $R$ would be trivial, and so each summand in the Steinitz decomposition of $M$ would be isomorphic to $R$.
	\end{remark}
	
	\begin{remark}\label{rmk: basis for Steinitz decomposition}
		We will sometimes need to express the Steinitz decomposition $M\simeq\fa_1\oplus\cdots\oplus\fa_n$ in terms of a basis. To do so, consider the $K$-vector space $M\otimes_RK$. Theorem~\ref{thm: Steinitz} guarantees the existence of $K$-linearly independent vectors $u_1,\ldots,u_m\in M\otimes_RK$ such that $M=\fa_1u_1+\cdots+\fa_mu_m \subset M \otimes_R K$. This explicitly realizes the Steinitz decomposition $M\simeq\fa_1\oplus\cdots\oplus\fa_m$. Unless stated otherwise, for any expression $M=\fa_1u_1+\cdots+\fa_mu_m$, we will always implicitly assume that the $u_i$ are vectors in $M\otimes_RK$ and that they are $K$-linearly independent. Given such an explicit decomposition of $M$, any other Steinitz decomposition can be expressed in terms of the $u_i$ as follows. Suppose that $\fb_1,\ldots,\fb_m$ are fractional ideals of $R$ such that $[\fa_1\cdots\fa_m]=[\fb_1\cdots\fb_m]$ in $\Cl(R)$. Then by Corollary~\ref{cor:steinitz} we have an isomorphism $M \simeq \fb_1\oplus\cdots\oplus\fb_m$ 
		of $R$-modules, so we can write $M = \fb_1 v_1+\cdots+\fb_mv_m \subset M \otimes_R K$ for some vectors $v_1, \dots, v_m \in M \otimes_R K$.
		The change-of-basis matrix taking the $u_i$ to the $v_i$ allows us to relate these two Steinitz decompositions by an element of $\GL_m(K)$.
	\end{remark}

	We finish this subsection by summarizing how standard multilinear operations interact with the Steinitz decomposition. Specifically, for two lattices $M=\fa_1u_1+\cdots+\fa_mu_m$ and $N=\fb_1v_1+\cdots+\fb_nv_n$, we may form the following lattices:
	\begin{enumerate}
		\item the \emph{tensor product} 
		\[M\otimes N=\bigoplus_{i=1}^m\bigoplus_{j=1}^n\fa_i\fb_j(u_i\otimes v_j);\]
		\item the \emph{symmetric powers} 
		\[\Sym^k(M)=\bigoplus_{1\leq i_1\leq\cdots\leq i_k\leq m}\fa_{i_1}\cdots\fa_{i_k}(u_{i_1}\cdots u_{i_k});\] 
		\item the \emph{exterior powers} 
		\[\exterior^{k}(M)=\bigoplus_{1\leq i_1<\cdots< i_k\leq m}\fa_{i_1}\cdots\fa_{i_k}(u_{i_1}\wedge\cdots\wedge u_{i_k});\] 
		\item the \emph{dual lattice} 
		\[M^*=\Hom(M,R)=\bigoplus_{i=1}^m\fa_i\inv u_i^*,\] where $u_i^*$ denotes the element of $(M\otimes_RK)^*$ such that $u_i^*(u_j)=\delta_{ij}$;
		\item and the \emph{lattice of homomorphisms} 
		\[\Hom(M,N)\simeq M^*\otimes_R N=\bigoplus_{i=1}^m\bigoplus_{j=1}^n\fa_i\inv\fb_j(u_i^*\otimes v_j).\] 
	\end{enumerate}

	\subsection{Algebras of finite rank over $R$} Next, we recall the theory of algebras over Dedekind domains.
	
	\begin{definition}\label{def: algebra}
		An \emph{algebra} of rank $n$ over $R$ is a commutative $R$-algebra $S$ that is also a rank-$n$ lattice when regarded as an $R$-module. Given a finite-dimensional \'etale $K$-algebra $L$, an \emph{order} in $L$ is an sub-$R$-algebra $S$ of $L$ with rank $\dim_K(L)$.
	\end{definition}
	
	Note that an order $S$ in an \'etale $K$-algebra $L$ has the property that $S\otimes_RK=L$. If $K$ is a number field and $L$ is a finite extension of $K$, then the definition of orders given in Definition~\ref{def: algebra} agrees with the usual definition of orders in number fields. Containment gives a partial ordering on the set of orders in an algebra $L$ over $K$. When $L/K$ is an extension of global fields, there is a unique \emph{maximal} order in which all other orders are properly contained, namely the ring of integers of $L$. In this section, we take $S$ to be an $R$-algebra of rank $n$ over $R$, and we write $L \defeq S \otimes_R K$ for its ambient $K$-algebra.
	
	\begin{remark}
		Let $S$ be an $R$-algebra of rank $n$. Because $R$ is integrally closed in $K$, the sublattice of $S$ generated by $1$ is maximal for its dimension and thus a direct summand of $S$ (i.e., the sublattice generated by 1 is \emph{primitive}). It follows that $S/R$ is an $R$-lattice of rank $n-1$. Therefore, we may write $S=R\oplus S/R$, although this decomposition is noncanonical. Moreover, we have an isomorphism $\exterior^n(S)\to\exterior^{n-1}(S/R)$ taking \[1\wedge v_1\wedge\cdots\wedge v_n\mapsto v_1\wedge\cdots\wedge v_n.\] 
	\end{remark}

	We now define ideals and ideal classes of algebras over $R$. For a more extensive background on orders and their ideals, we refer the reader to the comprehensive treatment in~\cite{MR4860678}.
	
	\begin{definition} \label{def: class group}
		We make the following standard definitions:
		\begin{itemize}
			\item A \emph{fractional ideal} $I$ of $S$ is a finitely generated $S$-submodule of $L$ that spans $L$ over $K$ (i.e., we have $I\otimes_RK=L$). Two fractional ideals $I$ and $J$ can be multiplied to form another fractional ideal, denoted $IJ$, given by finite sums of finite products of elements in $I$ and $J$.  Under this operation, the set of fractional ideals of $S$ forms a commutative monoid with identity given by the trivial ideal $S$. 
			\item Define an equivalence relation on the monoid of fractional ideals of $S$ by setting fractional ideals $I$ and $I'$ equivalent if $I = \alpha I'$ for some $\alpha \in L^\times$. The equivalence classes of this relation are called \emph{ideal classes}, and ideal multiplication descends to a multiplication law on ideal classes. Under this law, the ideal classes form a commutative monoid called \emph{the ideal class monoid}. 
			\item A fractional ideal $I$ is \emph{invertible} if there exists a fractional ideal $J$ such that $IJ=S$. The submonoid of invertible ideals of $S$ forms a group called the \emph{ideal group}, denoted $\mathcal{I}(S)$. Likewise, invertible ideal classes of $S$ form a group called the \emph{$($ideal$)$ class group} of $S$, denoted $\Cl(S)$, which is simply the quotient of $\mathcal{I}(S)$ by the subgroup of principal ideals, i.e., ideals of the form $\alpha S$ for $\alpha\in L^\times$.
		\end{itemize}
	\end{definition}
	
	Recall from Definition~\ref{def: algebra} and Remark~\ref{rmk: basis for Steinitz decomposition} that we may write $S=R+\fa\xi$ for some ideal $\fa$ in the Steinitz class of $S$ and for some $\xi\in L$. Since $S$ is quadratic over $R$, we have that $L=K(\xi)=K[\xi]/(\xi^2-t\xi+u)$, which gives us a multiplication law $\xi^2=t\xi-u$ for $\xi$, where $t\in\fa\inv$ and $u\in\fa^{-2}$. The algebra $S$ is determined by the ideal $\fa$ and this multiplication law, and we may write $S\simeq R[\fa\xi]/(\fa^2(\xi^2-t\xi+u))$, which is a subring of $K[\xi]/(\xi^2-t\xi+u)$. Note that the generator $\xi$ is unique up to the change-of-coordinates $\xi\mapsto\lambda\xi+a$, where $\lambda\in R^\times$ and $a\in\fa\inv$. Alternatively, $S$ is determined by its norm map, \[N_{S/R}\colon S\to R,\quad x+y\xi\mapsto x^2+txy+uy^2,\] which is simply another way of packaging the same data of the two numbers $t$ and $u$. A more detailed account of the theory of quadratic algebras over $R$ is given in \S\ref{sec-bqfs} and \S\ref{sec-paramquad}.

	\subsection{Module indices and ideal norms} Two lattices of the same rank over $R$ differ only by their Steinitz classes. The \emph{module index} of two lattices gives a more precise way of quantifying the extent to which two lattices over $R$ differ:
	\begin{definition}\label{def: module index}
		Let $V$ be an $n$-dimensional $K$-vector space, and let $M$ and $N$ be $R$-lattices in $V$ of rank $n$. Let $\fp$ denote a prime of $R$. Since $R$ is Dedekind, the localization $R_{(\fp)}$ is a PID, and we have $M_{(\fp)}\simeq R_{(\fp)}^n\simeq N_{(\fp)}$. Let $\phi_{(\fp)}\colon M_{(\fp)}\to N_{(\fp)}$ be an isomorphism of $R_{(\fp)}$-modules, and let $\wh{\phi}_{(\fp)}$ be the unique extension of $\phi_{(\fp)}$ to an isomorphism $V\to V$. We define the local index at $\fp$ to be the principal fractional ideal of $R_{(\fp)}$ given by \[[M_{(\fp)}:N_{(\fp)}]_{R_{(\fp)}} \defeq \det(\wh{\phi}_{(\fp)})R_{(\fp)}.\] Note that this definition does not depend on our choice of $\phi_{(\fp)}$, since any other isomorphism $M_{(\fp)}\to N_{(\fp)}$ can be written as $\phi_1\phi_{(\fp)}\phi_2$, where $\phi_1\colon M_{(\fp)}\to M_{(\fp)}$ and $\phi_2\colon N_{(\fp)}\to N_{(\fp)}$ are $R_{(\fp)}$-module automorphisms and hence have unit determinants. Then define the \emph{module index} of $N$ in $M$ to be 
		$$[M:N]_R\defeq \bigcap_{\fp}\,[M_{(\fp)}:N_{(\fp)}]_{R_{(\fp)}}.$$
		where the intersection takes place in $K$ and ranges over the primes of $R$. 
	\end{definition}
	\begin{lemma} \label{lem-modinddefine}
		With notation as in Definition~\ref{def: module index}, we have that $[M : N]_R$ is the nonzero fractional ideal of $R$ whose localization at each prime $\fp$ is equal to the local index $[M_{(\fp)}:N_{(\fp)}]_{R_{(\fp)}}$.
	\end{lemma}
	\begin{proof}
		To prove that $[M : N]_R$ is a nonzero fractional ideal of $R$, it suffices to show that $[M_{(\fp)} : N_{(\fp)}]_{R_{(\fp)}} = R_{(\fp)}$ for all but finitely many primes $\fp$, for which it suffices to show that there exists an isomorphism $\phi \colon V \to V$ of $K$-vector spaces such that the map $\phi_{(\fp)}$ obtained by restricting $\phi$ to $M_{(\fp)}$ defines an isomorphism $\phi_{(\fp)} \colon M_{(\fp)} \to N_{(\fp)}$ for all but finitely many primes $\fp$. For this, choose Steinitz decompositions $M \simeq \fa u \oplus \bigoplus_{i = 1}^{n-1} Ru_i$ and $N \simeq \fb v \oplus \bigoplus R v_i$, where $u, u_1, \dots, u_{n-1}$ and $v, v_1, \dots, v_{n-1}$ are two bases of $L$ over $K$.  We can take $\phi$ to be defined by the element of $\on{GL}_n(K)$ sending the basis $u, u_1, \dots, u_{n-1}$ to the basis $v, v_1, \dots, v_{n-1}$, for then the restriction of $\phi$ to every prime $\fp \nmid \fa\fb$ defines an isomorphism $\phi_{(\fp)} \colon M_{(\fp)} \to N_{(\fp)}$, as desired. That the localization of $[M : N]_R$ at $\fp$ is the local index at $\fp$ now follows from, e.g.,~\cite[proof of Proposition~I.12.6]{Neukirch}.
	\end{proof}

	The module index interacts with fractional ideals through the notion of ideal norm. Indeed, we have a map $N_{S/R}\colon \mathcal{I}(S)\to\mathcal{I}(R)$, called the \emph{ideal norm}, taking an ideal $I$ in $ \mathcal{I}(S)$ to $[S:I]_R$ (in the language of Definition~\ref{def: module index}, we think of $S$ and $I$ as being lattices over $R$ in the $K$-vector space $L$). The following lemma shows that a proper invertible integral ideal $I \subsetneq S$ has nontrivial integral ideal norm. 
	\begin{lemma}\label{proper cont upstairs implies proper cont downstairs}
		Let $I \in \mc{I}(S)$ be such that $I\subsetneq S$. Then $N_{S/R}(I) \subsetneq R$. 
	\end{lemma}
	\begin{proof}
		Recall that \[[S:I]_R=\bigcap_{\fp}\,[S_{(\fp)}:I_{(\fp)}]_{R_{(\fp)}},\] where the intersection runs over all prime ideals $\fp$ of $R$. Note that because $I \subset S$ is integral, we have that $I_{(\fp)} \subset S_{(\fp)}$ is integral for each $\fp$, implying that $[S : I]_R \subset \bigcap_{\fp} R_{(\fp)} = R$ is integral.
		
		Next, since $I$ is properly contained in $ S$, there exists some prime $\fP$ of $S$ such that $I\subset\fP\subsetneq S$. Let $\fp$ denote the prime of $R$ lying under $\fP$, and consider the exact sequence \[\begin{tikzcd}
			0\rar&\fP\rar&S\rar&S/\fP\rar&0.
		\end{tikzcd}\]
		Localization is exact, so localizing away from $\fp$ yields the following exact sequence: \[\begin{tikzcd}
			0\rar&\fP_{(\fp)}\rar&S_{(\fp)}\rar&(S/\fP)_{(\fp)}\rar&0.
		\end{tikzcd}\]
		The quotient $S/\fP$ is an integral domain, so the localization $(S/\fP)_{(\fp)}$ is nonzero. It follows that $\fP_{(\fp)}$ is a proper submodule of $S_{(\fp)}$, and the inclusion $I_{(\fp)}\subset\fP_{(\fp)}$ implies that $I_{(\fp)}$ is properly contained in $S_{(\fp)}$. Thus, $[S_{(\fp)}:I_{(\fp)}]_{R_{(\fp)}}$ is a proper integral ideal of $R_{(\fp)}$, i.e., $[S_{(\fp)}:I_{(\fp)}]_{R_{(\fp)}}\subset\fp R_{(\fp)}$. Thus, $[S:I]_R\subsetneq R$.
	\end{proof}
	
	We may extend the ideal norm $N_{S/R}$ to the zero ideal by setting $N_{S/R}(0)=0$. It is natural to ask how our notion of ideal norm compares to the usual notion of norms on \'{e}tale algebras. Indeed, recall that because $L$ is finite over $K$, we have a map $N_{L/K}\colon L\to K$ taking $\alpha\in L$ to the determinant of the $K$-linear map of multiplication by $\alpha$. The following lemma confirms that the restriction of the ideal norm to principal ideals agrees with the element norm.
	\begin{lemma}\label{ideal norm agrees with element norm}
		We have that $N_{S/R}(\alpha S)=N_{L/K}(\alpha)R$ for every $\alpha\in L$. 
	\end{lemma}
	\begin{proof}
		If $\alpha=0$, then the statement is trivial, so we may assume otherwise. Then
		\begin{align*}
			N_{S/R}(\alpha S)=[S:\alpha S] & =\bigcap_{\fp}\,[S_{(\fp)}:\alpha S_{(\fp)}]_{R_{(\fp)}} \\
			& =\bigcap_{\fp} \left(\det\left(S_{(\fp)}\overset{\times\alpha}{\longrightarrow}S_{(\fp)}\right)\right)R_{(\fp)}  =\left(\det\left(L\overset{\times\alpha}{\longrightarrow}L\right)\right)R=N_{L/K}(\alpha)R,
		\end{align*}
		where the intersections run over all prime ideals $\fp$ of $R$.  
	\end{proof}

	From Lemma~\ref{ideal norm agrees with element norm}, it follows immediately that the restriction of the ideal norm to the invertible ideals of $S$ is a group homomorphism, as we now show:
	\begin{lemma}\label{ideal norm is multiplicative}
		The ideal norm $N_{S/R}\colon \mathcal{I}(S)\to\mathcal{I}(R)$ is a group homomorphism. 
	\end{lemma}
	\begin{proof}
		By~\cite[Proposition~I.12.4]{Neukirch}, every element of $\mathcal{I}({S_{(\fp)}})$ is principal of the form $\alpha S_{(\fp)}$ for some $\alpha\in L^\times$. By Lemma~\ref{ideal norm agrees with element norm}, we know that $N_{S_{(\fp)}/R_{(\fp)}}$ is a group homomorphism, since $N_{L/K}$ is. Now, for $I,J\in\mathcal{I}(S)$, we have that \[N_{S/R}(IJ)=\bigcap_{\fp} N_{S_{(\fp)}/R_{(\fp)}}(I_{(\fp)} J_{(\fp)})=\bigcap_{\fp} N_{S_{(\fp)}/R_{(\fp)}}(I_{(\fp)})N_{S_{(\fp)}/R_{(\fp)}}(J_{(\fp)})=N_{S/R}(I)N_{S/R}(J), \]
		where the last step above follows from Lemma~\ref{lem-modinddefine} and~\cite[Proposition~I.12.6]{Neukirch}.
	\end{proof}
	
	If $S$ is a quadratic $R$-algebra with Steinitz class $\St(S)=[\fa]$, then writing $S=R+\fa\xi$ and $I=\fb_1\beta_1+\fb_2\beta_2$, we claim that \begin{equation}\label{eqn: explicit ideal norm}
		[S:I]_R=\det(g)\fa\inv\fb_1\fb_2,
	\end{equation} where $g$ is the change-of-basis matrix in $\GL_2(K)$ taking $1,\xi$ to $\beta_1,\beta_2$. By the last paragraph of Definition~\ref{def: module index}, it suffices to verify \eqref{eqn: explicit ideal norm} locally. For each prime $\fp$ of $R$, the map $g$ gives us an isomorphism $g\colon S_{(\fp)}\to I_{(\fp)}$ of $R_{(\fp)}$-modules, where \[S_{(\fp)}=R_{(\fp)}+a\xi R_{(\fp)}\quad\textrm{and}\quad I_{(\fp)}=\beta_1b_1R_{(\fp)}+\beta_2b_2R_{(\fp)}\] and $aR_{(\fp)}=\fa_{(\fp)}$ and $b_iR_{(\fp)}=(\fb_{1})_{(\fp)}$ for $i \in \{1,2\}$. We see that \[[S_{(\fp)}:I_{(\fp)}]_{R_{(\fp)}}=\det(g)a\inv b_1b_2 R_{(\fp)},\] from which \eqref{eqn: explicit ideal norm} follows. 
	
	An immediate consequence of \eqref{eqn: explicit ideal norm} is that the ideal norm of an ideal $I$ with $\St(I)=\St(S)$ is principal. It turns out that that this property is characteristic. Indeed, if $[S:I]_R$ is principal, then \eqref{eqn: explicit ideal norm} implies that $\fa$ and $\fb=\fb_1\fb_2$ belong to the same ideal class. That is, $S$ and $I$ have the same Steinitz class. We summarize this crucial property in the following lemma. 
	\begin{lemma}\label{lem: rel tors Steinitz}
		Let $I$ be a fractional ideal of $S$. Then $S$ and $I$ have the same Steinitz class if and only if $[S:I]_R$ is a principal ideal. 
	\end{lemma}
	
	\subsection{Orientations} \label{subsec: orientations}

	We will have occasion to treat the Steinitz class of a lattice over $R$ as an ideal of $R$ (rather than as an ideal class), to which end we make the following definition:
	\begin{definition}
		Let $M$ be a rank-$m$ lattice over $R$. For a fractional ideal $\fa$ of $R$ representing $\St(M)$, an $\fa$-\emph{orientation} of $M$ is a choice of isomorphism $\pi \vcentcolon \exterior^m(M)\overset{\sim}\longrightarrow\fa$. For each fractional ideal $\fa$, the set of $\fa$-orientations of $M$ are in (noncanonical) bijection with $R^\times$. Given a decomposition $M=\fa_1u_1\oplus\cdots\oplus\fa_mu_m$, where $\fa_1\cdots\fa_m=\fa$, and given $u\in R^\times$ a unit, we can specify an $\fa$-orientation $\pi$ by setting \[\pi(x_1u_1\wedge\cdots\wedge x_mu_m)=ux_1\cdots x_m.\] An \emph{orientation} of $M$ is an $\fa$-orientation of $M$ for some fractional ideal $\fa$ representing $\St(M)$.
	\end{definition}
	
	In this paper, we are solely concerned with quadratic algebras over $R$. For quadratic algebras, orientations are of particular interest because the data of an orientation is equivalent to distinguishing a certain basis of the algebra. Recall that an algebra $S$ over $R$ has a noncanonical decomposition as $S=R\oplus S/R$, which tells us that $\exterior^n(S)\simeq\exterior^{n-1}(S/R)$. It follows that an orientation on $S$ is equivalent to an orientation on $S/R$. When $S$ is quadratic, $S/R$ is an $R$-module of rank 1. Emulating the procedure outlined in Remark~\ref{rmk: basis for Steinitz decomposition} tells us that $S$ can be written as $S=R+\fa\xi$ for some element $\xi$ of $L$ independent of $1$ and $\fa$ a representative of the Steinitz class of $S$. Decomposing $S$ in this way, however, requires making several choices. Most obvious is the choice of Steinitz class representative, but, once $\fa$ has been chosen, we also have many options for $\xi$, as any $R^\times$-multiple of $\xi$ yields the exact same decomposition. 
	
	Equipping $S$ with an $\fa$-orientation $\pi\colon\exterior^2(S)\to\fa$ solves this issue. The orientation specifies a choice of Steinitz class representative $\fa$ as well as a choice of $\xi$, as we may take $\xi$ to be the element of $L$ such that $S=R+\fa\xi$ and $\pi_K(1\wedge\xi)=1$, where $\pi_K$ denotes the unique extension of $\pi$ to an isomorphism $\pi_K\colon\exterior^2 (L)\to K$. Thus, for any ideal $\fa$ of $R$, fixing an $\fa$-orientation $\pi\colon \exterior^2(S)\to\fa$ of $S$ is equivalent to distinguishing the basis $1,\xi$ of $S$ such that $S=R+\fa\xi$ and $\pi_K(1\wedge\xi)=1$. Note that $\xi$ is unique up to translation by $\fa\inv$. In other words, we may replace $\xi$ by $\xi+a$ for any $a\in\fa\inv$ and end up with the same oriented quadratic ring; any oriented quadratic ring isomorphic to $S$ can be written as $R[\fa\xi]/(\fa^2(\xi'^2-t'\xi'+u'))$ for $\xi'=\xi+a$, where $a$ is an element of $\fa\inv$.

	We further elucidate the relationship between orientations and the ideal norm with the following definition.
	
	\begin{definition}\label{def: oriented norm}
		Let $S$ be a quadratic $R$-algebra oriented by $\pi\colon \exterior^2(S)\to\fa$. Writing $S=R+\fa\xi$, we may suppose that $\pi_K(1\wedge\xi)=1$, where, as before $\pi_K$ denotes the unique extension of $\pi$ to an isomorphism $\pi_K\colon \exterior^2 (L)\to K$. For any ideal $I$ of $S$ with $\St(I)=[\fa]=\St(S)$, write $I=R\alpha+\fa\beta$ for elements $\alpha,\beta\in L$ (which we may do by Remark~\ref{rmk: basis for Steinitz decomposition}). Let $g$ denote the change-of-basis-matrix in $\GL(R\oplus\fa)$ relating $1,\xi$ to $\alpha,\beta$. The \emph{based ideal norm} of the decomposition $I=R\alpha+\fa\beta$ is defined to be $\det(g)$ and is denoted $N(I;\pi,\alpha,\beta)$.
	\end{definition}
	
	This definition is reminiscent of Definition~\ref{def: module index}, and we relate the two in the following. 
	\begin{lemma}\label{ideal norm and steinitz index agree}
		Let $S$ be a quadratic ring extension of $R$, oriented by $\pi\colon \exterior^2(S)\to\fa$. Let $I$ a fractional ideal of $S$ with Steinitz class $[\fa]$. Then, for any Steinitz decomposition of the form $I=R\alpha+\fa\beta$, we have that $[S:I]_R=N(I;\pi,\alpha,\beta) R$.
	\end{lemma}
	\begin{proof}
		We have specified Steinitz decompositions $S=R+\fa\xi$ and $I=R\alpha+\fa\beta$. Letting $g$ be the change-of-basis matrix sending $1,\xi\mapsto\alpha,\beta$, \eqref{eqn: explicit ideal norm}, we have that \[[S:I]_R=\det(g)\fa\inv\fb_1\fb_2=\det(g)R=N(I;\pi,\beta_1,\beta_2)R,\] as desired.
	\end{proof}

	We reiterate that the based ideal norm depends on the Steinitz decomposition of the ideal $I=R\alpha+\fa\beta$ and the orientation of the ring $S$. Any two such Steinitz decompositions of $I$ are related by an element $h$ of $\GL(R\alpha+\fa\beta)$, whose determinant lies in $R^\times$, and the corresponding element ideal norms differ by $\det(h)$. Thus, the based ideal norms of two distinct such Steinitz decompositions of $I$ are equal if and only if the corresponding Steinitz decompositions can be related by an element of $\SL(R\alpha+\fa\beta)$.

	\subsection{Discriminants}\label{subsec: disc}
	One of the most important invariants of an algebra over $R$ is its \emph{discriminant}. When the characteristic of $R$ is not equal to 2, isomorphism classes of quadratic algebras are determined by their discriminants. There are several possible notions of discriminant, all of which are closely related. For a more complete treatment of this theory, we refer the reader to \cite{CohenBook}.
	
	We begin with a definition that depends on the choice of Steinitz decomposition of $S$. First, recall that we have a trace map $T\colon S\to R$ taking an element $\alpha\in S$ to the trace of the $R$-linear map of multiplication by $\alpha$. This map extends to a trace pairing $T\colon S\otimes_RS\to R$ given by $x\otimes y\mapsto T(xy)$. 
	\begin{definition}\label{def: disc based}
		Let $S$ be an algebra of rank $n$ with Steinitz decomposition $S=\fa_1\xi_1+\cdots+\fa_n\xi_n$. The \emph{based discriminant} of $S$ with respect to the Steinitz decomposition $(\fa_i,\xi_i)_{1\leq i\leq n}$ is defined to be the determinant of the matrix $T(\xi_i\xi_j)$. We denote this based discriminant by $\disc(S;\fa_i,\xi_i)$, which is an element of $(\fa_1\cdots\fa_n)^{-2}$.
	\end{definition}
	
	For example, given a quadratic ring $S=R[\fa\xi]/(\fa^2(\xi^2+t\xi+u))$ with $t\in\fa\inv$ and $u\in\fa^{-2}$, the based discriminant of $S$ with respect to the decomposition $S=R+\fa\xi$ is computed to be $t^2-4u\in\fa^{-2}$. While based discriminants facilitate computation, their drawback is that their definition depends on the choice of Steinitz decomposition of $S$. Thus, we also make the following more equivariant definition, which generalizes Definition~\ref{def: disc based}.
	\begin{definition}
		The \emph{discriminant} of an algebra $S$ of rank $n$ over $R$, denoted $\disc(S)$, is the determinant of the trace pairing $S\otimes_R S\to R$ which we note is an element of $(\exterior^n(S))^{\otimes-2}$, where here we are taking $(\exterior^n(S))^{\otimes-2}$ to denote the second tensor power of the dual of $\exterior^n(S)$ as an $R$-module. We may view the discriminant as a pair $(N,d)$, where $N$ is a locally free $R$-module of rank 1 and $d\in N^{\otimes2}$. Two discriminants $(N,d)$ and $(N',d')$ are said to be isomorphic if there is an isomorphism $N\to N'$ of $R$-modules taking $d\mapsto d'$.
	\end{definition}
	
	Finally, we define the \emph{discriminant ideal} of an algebra $S$ over $R$---a notion that will agree with the usual notion of discriminant ideal from algebraic number theory when $S$ and $R$ are the rings of integers in extensions of global fields. While our initial definition depends on the Steinitz decomposition of $S$, we will see that the discriminant ideal is basis independent.
	\begin{definition}
		Let $S$ be an algebra of rank $n$ with Steinitz decomposition $S=\fa_1\xi_1+\cdots+\fa_n\xi_n$. The \emph{discriminant ideal} of $S$ with respect to the Steinitz decomposition $(\fa_i,\xi_i)_{1\leq i\leq n}$ is defined to be \[\Disc_{S/R}(\fa_i,\xi_i)=\fa_1^2\cdots\fa_n^2\disc(S;\fa_i,\xi_i),\] which we note is an integral ideal of $R$.
	\end{definition}
	\begin{theorem}[{\cite[Corollary~1.4.3]{CohenBook}}]
		Let $S$ be an $R$-algebra of rank $n$ with Steinitz decompositions $(\fa_i,\xi_i)_{1\leq i\leq n}$ and $(\fb_i,\omega_i)_{1\leq i\leq n}$. Then \[\Disc_{S/R}(\fa_i,\xi_i)=\Disc_{S/R}(\fb_i,\omega_i).\] 
	\end{theorem}
	
	Since the $\Disc_{S/R}(\fa_i,\xi_i)$ does not depend on the Steinitz decomposition, we call it the \emph{discriminant ideal} of the extension $S/R$ and denote it by $\Disc_{S/R}$. When $S$ and $R$ are rings of integers in an extension of global fields, it is straightforward to verify that the discriminant ideal, as we have defined it, agrees with the usual notion of discriminant ideal from algebraic number theory.
	
	When $S$ is a quadratic algebra, recall that an orientation $\pi\colon\exterior^2(S)\to\fa$ on $S$ is equivalent to a choice of basis $S=R+\fa\xi$ with 1 as the first basis vector. Thus, when we refer to the discriminant of an oriented quadratic algebra $(S,\pi)$, we always mean its based discriminant with respect to the basis specified by the orientation, which we regard as a genuine element of $\fa^{-2}$.

	\subsection{The class group}\label{subsec: class group}

	Let $S$ be an algebra of rank $n$ over $R$. By Lemma~\ref{ideal norm is multiplicative}, we have a group homomorphism $N_{S/R}\colon \mathcal{I}(S)\to\mathcal{I}(R)$; furthermore, by Lemma~\ref{ideal norm agrees with element norm}, this map also descends to a homomorphism \(N_{S/R}\colon \Cl(S)\to\Cl(R).\)
	\begin{definition}
		The kernel of the map $N_{S/R}\colon \Cl(S)\to\Cl(R)$ is called the \emph{relative class group} and is denoted $\Cl(S/R) \defeq \ker (N_{S/R}:\Cl(S)\to\Cl(R))$. 
	\end{definition}
	
	The following result demonstrates how, under certain conditions, the torsion subgroups of the relative class group $\on{Cl}(S/R)$ are related to those of $\on{Cl}(S)$ and $\on{Cl}(R)$. 
	\begin{proposition} \label{prop-classsplits}
		Let $S$ be an algebra of rank $n$ over $R$, and let $\ell \in \mathbb{N}$ be a prime number not dividing $n$. Then we have that
		\[\Cl(S)[\ell]\simeq\Cl(R)[\ell]\oplus\Cl(S/R)[\ell].\] 
	\end{proposition}
	\begin{proof}
		We start by proving the following lemma about the group homomorphism $\iota\colon\mathcal{I}(R)\to\mathcal{I}(S)$ that sends a fractional ideal $\fa$ to its extension $\fa S$.
		\begin{lemma}\label{norm of ideal of R in S}
			For an ideal $\fa\in\mathcal{I}(R)$, we have $N_{S/R}(\iota(\fa))=N_{S/R}(\fa S)=\fa^n$.
		\end{lemma}
		\begin{proof}[Proof of Lemma~\ref{norm of ideal of R in S}]
			Choose a Steinitz decomposition \(S=\fb\xi\oplus\bigoplus_{i=1}^{n-1}R\xi_i,\) where $\xi,\xi_1,\ldots,\xi_{n-1}$ form a basis of $L$ over $K$. Then \(\fa S=\fa\fb\xi\oplus\bigoplus_{i=1}^{n-1}\fa\xi_i;\) the lemma follows from computing that $[S_{(\fp)}:(\fa S)_{(\fp)}]_{R_{(\fp)}}=\fa_{(\fp)}^n$ for every prime $\fp$ of $R$.  
		\end{proof} 
		
		Now, let $d \in \mathbb{N}$ such that $dn \equiv 1 \pmod \ell$. Composing $\iota$ with the $d^{\mathrm{th}}$ power of the ideal norm gives us  \[\begin{tikzcd}
			\Cl(R)[\ell]\arrow[r,"\iota"] & \Cl(S)[\ell]\arrow[r,"N_{S/R}^d"] &[1em]\Cl(R)[\ell].
		\end{tikzcd}\]
		For convenience, let $N_{S/R,\ell}$ denote the restriction of the norm map to $\Cl(S)[\ell]$, and note that we have $\ker (N_{S/R,\ell}) = \ker (N_{S/R,\ell}^d)$. By Lemma~\ref{norm of ideal of R in S}, the image of an ideal class $[\fa]\in\Cl(R)[\ell]$ of $R$ under these maps is \[[N_{S/R}(\fa S)]^d=[\fa^n]^d=[\fa]^{dn}=[\fa],\] so $\iota$ and $N_{S/R}^d$ are inverses. Thus, we have a splitting of the exact sequence \
		\begin{equation} \label{eq-classsplit}
			\begin{tikzcd}
				1\arrow[r] & \Cl(S/R)[\ell]\arrow[r] & \Cl(S)[\ell]\arrow[r,"N_{S/R}^d"]&\Cl(R)[\ell]\arrow[r]\arrow[l, bend left=33,"\iota", dashed] & 1,
			\end{tikzcd}
		\end{equation}
		where we note that $N_{S/R}^d$ is surjective because $N_{S/R}^d\circ\iota=\id$. It follows that we may write $\Cl(S)[\ell]$ as the direct sum of $\Cl(R)[\ell]$ with $\ker N_{S/R,\ell}^d=\ker N_{S/R,\ell}$. This completes the proof of Proposition~\ref{prop-classsplits}.
	\end{proof}
	Consequently, if we take the $\ell$-torsion subgroup in $\on{Cl}(R)$ as given, determining the $\ell$-torsion subgroup in $\Cl(S)$ amounts to doing the same for $\Cl(S/R)$. In what follows, we will take $n = 2$ and $\ell = 3$ and give a parametrization of the relative $3$-torsion subgroup $\on{Cl}(S/R)[3]$.
	
	\begin{remark}
		Because $n$ is invertible modulo powers of $\ell$, the exact sequence~\eqref{eq-classsplit} still holds when the $\ell$-torsion subgroups are replaced by their $\ell^\infty$-torsion counterparts and when the map $N_{S/R}^d$ is replaced by $N_{S/R}^{1/d}$.
	\end{remark}
	
	We remark briefly about 3-torsion ideals in the ideal group. Let $I$ be a fractional ideal of $S$ such that $I^3=S$. If $S$ is a maximal order, then by unique factorization into products of prime ideals, we have that $I^3=S$ forces $I=S$. In other words, when $S$ is maximal, $\cali(S)[3]=\varnothing$. Surprisingly, this is not always the case; consider the following example.
	\begin{example}({\cite[Example 5]{BV}})
		Let $R=\Z$ and $S=\Z[\sqrt{-11}]$. The ideal $I=(2,(1-\sqrt{-11})/2)$ is a fractional ideal of $S$ with norm 1. Because $I^3\subset S$, it follows (from the fact that the ideal has norm 1) that $I^3=S$, despite the fact that $I\neq S$. Thus, the ideal $I$ has order 3 in both $\cali(S)$ and in $\Cl(S)$.
	\end{example}
	
	Thus, it is possible for there to exist an element of $\Cl(S/R)[3]$ that has order 3 simply because it is equivalent to an ideal that has order 3 in the ideal group! 
	
	Now, suppose that $R$ is the ring of integers of some global field $K$, and further assume that $S$ is an order in some finite extension $L/K$. Let $\cals$ denote the integral closure of $S$ in $L$. We know that $\cals$ is a Dedekind domain with $\Cl(L)\simeq\Cl(\cals)$ (see, e.g., \cite[Proposition~I.12.8]{Neukirch}). Then the class group of $S$ is related to that of $\cals$ via the following canonical exact sequence: \begin{equation}\label{eqn: class gp exact seq max order}
		\begin{tikzcd}
			1\rar&S^\times\rar&\cals^\times\rar&\bigoplus_{\fp}\cals_{(\fp)}^\times/S_{(\fp)}^\times\rar&\Cl(S)\rar&\Cl(\cals)\rar&1,
		\end{tikzcd}
	\end{equation} where the sum runs over the prime ideals $\fp$ of $S$ and where $\cals_{(\fp)}$ denotes the integral closure of $\cals_{(\fp)}$ in $L$. While we do not use \eqref{eqn: class gp exact seq max order} in the sequel, it is fundamental and stated for the sake of completeness. For a proof of \eqref{eqn: class gp exact seq max order}, we refer the reader to \cite[Proposition~I.12.9]{Neukirch}.

	\section{Parametrization of quadratic orders and Gauss composition}\label{sec: param of quad orders and gauss comp}
	
	To prove Theorem~\ref{thm:mainresult}, we require certain background material concerning the relationship between quadratic rings over $R$, ideal classes of such rings, and binary quadratic forms over $R$. The purpose of this section is to introduce this background material. We begin in \S\ref{sec-bqfs} with a discussion of binary quadratic forms over $R$. Then, in \S\ref{sec-paramquad}, we recall the parametrization of quadratic rings by monic binary quadratic forms. While this parametrization is due to O'Dorney~\cite{ODorney}, we have reexpressed it in a way that best suits our desired application. Finally, in \S\ref{subsec: Gauss composition}, we present the parametrization of ideal classes of oriented quadratic extensions of $R$ by binary quadratic forms---i.e., a generalization of Gauss composition that works over Dedekind domains---which is due to Wood~\cite{WoodGaussComp} (cf.~O'Dorney~\cite{ODorney}). For a more complete picture of this beautiful theory, we refer the reader to the respective papers by Bhargava, Wood, and O'Dorney~\cite{BhargavaHCLI,WoodGaussComp,ODorney}. 
	
	\subsection{Binary quadratic forms} \label{sec-bqfs}
	We begin by giving Wood's definition of binary quadratic forms over a Dedekind domain \cite{WoodGaussComp}. While this definition may seem quite abstract initially, it has an explicit interpretation that facilitates computation and agrees with the usual notion of a binary quadratic form over the integers.
	
	\begin{definition}\label{def: BQF}
		A \emph{binary quadratic form} over $R$ is the data of a locally free $R$-module $M$ of rank 2 and an element $f\in\Sym^2(M^*)$. Two binary quadratic forms $(M,f)$ and $(M',f')$ are said to be \emph{isomorphic} if there exists an isomorphism $M\to M'$ taking $f\mapsto f'$. A \emph{linear binary quadratic form} is the data of a locally free $R$-module $M$ of rank 2, a locally free $R$-module $L$ of rank 1, and an element $f\in\Sym^2(M^*)\otimes L$. Writing $M=\fa u\oplus\fb v$ and $L=\mathfrak{c}z$, we may express $f$ as \[f(xu,yv)=px^2z+qxyz+ry^2z,\] where $p\in\fa^{-2}\mathfrak{c}$, $q\in\fa\inv\fb\inv\mathfrak{c}$, and $r\in\fb^{-2}\mathfrak{c}$. Two linear binary quadratic forms $(M,L,f)$ and $(M',L',f')$ are said to be \emph{isomorphic} if there exist isomorphisms $M\to M'$ and $L\to L'$ such that $f\mapsto f'$ under the induced isomorphism $\Sym^2(M^*)\otimes L\to\Sym^2((M')^*)\otimes L'$. The \emph{discriminant} of a linear binary quadratic form $f\in\Sym^2(M^*)\otimes L$ is the element of $(\exterior^2(M)\otimes L)^{\otimes2}$ given by $(q^2-4pr)((u\wedge v)\otimes z)^2$ for Steinitz decompositions $M=\fa u\oplus\fb v$ and $L=\mathfrak{c} z$ as above. If we set $L=\exterior^2(M^*)\otimes\mathfrak{c}$ for some ideal $\mathfrak{c}$ of $R$, then we may regard the discriminant simply as an element in $\mathfrak{c}^{2}$.
		
		For a prime $\fp$ of $R$, suppose that $M_{(\fp)}=R_{(\fp)} x+R_{(\fp)} y$ and $L=R_{(\fp)} z$ so that $f$ is given locally by $ax^2z+bxyz+cy^2z$. We say that $f$ is \emph{primitive} if $a$, $b$, and $c$ generate the unit ideal in $R_{(\fp)}$ for all primes $\fp$ of $R$, i.e., if $aR_{(\fp)}+bR_{(\fp)}+cR_{(\fp)}=R_{(\fp)}$ for all primes $\fp$ of $R$. If $M=\fa x\oplus\fb v$ and $L=\mathfrak{c}z$ with \[f(xu,yv)=px^2z+qxyz+ry^2z,\] then we see that $f$ is primitive if and only if \[p(\fa^{2}\mathfrak{c}\inv)_{(\fp)}+q(\fa\fb \mathfrak{c}\inv)_{(\fp)}+r(\fb^{2}\mathfrak{c}\inv)_{(\fp)}=R_{(\fp)}\] for all primes $\fp$ of $R$. 
	\end{definition}
	
	\begin{remark}\label{rmk: primitive defs are same}
		The definition of primitivity given in Definition~\ref{def: BQF}, which is taken from Wood \cite{WoodGaussComp}, is not at first glance the same as the definition of primitivity used by O'Dorney \cite{ODorney}. However, these definitions are indeed equivalent. For a binary quadratic form $(M,L,f)$, the \emph{image} of $f$ is the smallest rank-1 submodule $N\subset L$ such that $f$ lies in the image of the natural injection $\Sym^2(M^*)\otimes N\hookrightarrow\Sym^2(M^*)\otimes L$. O'Dorney defines $f$ to be primitive if it does not factor through any proper sublattice of $L$, i.e., if the image of $f$ is $L$. If $M=\fa u\oplus\fb v$ and $f(xu,yv)=px^2z+qxyz+ry^2z$, then the image of $f$ is $(p\fa^2+q\fa\fb+r\fb^2)z\subset\mathfrak{c}z$, so $f$ is primitive if and only if \[p\fa^2+q\fa\fb+r\fb^2=\mathfrak{c}.\] The latter condition can be verified locally---$\fa^2+\fa\fb+\fb^2=\mathfrak{c}$ if and only if $\fa_{(\fp)}^2+\fa_{(\fp)}\fb_{(\fp)}+\fb_{(\fp)}^2=\mathfrak{c}_{(\fp)}$ for each prime $\fp$ of $R$. By the above, it follows that O'Dorney's formulation of primitivity is equivalent to that of Wood.
	\end{remark}

	\subsection{The parametrization} \label{sec-paramquad}

	Fix a Dedekind domain $R$ with fraction field $K$, and also fix a fractional ideal $\fa$ of $R$. In this section, we construct a parametrization of quadratic rings over $R$ with Steinitz class $[\fa]$ by certain binary quadratic forms. Recall from \S\ref{subsec: orientations} that, for any quadratic ring $S$ over $R$ with Steinitz class represented by $\fa$, we may write $S = R+\fa\xi$ for some $\xi$ satisfying $\xi^2+a\xi+b = 0$ with $a \in \fa\inv$ and $b \in \fa^{-2}$; here, the choice of $\xi$ is well-defined up to $\xi \mapsto \lambda\xi+u$, where $\lambda \in R^\times$ and $u \in \fa\inv$.

	On the other side of the parametrization is the set of binary quadratic forms $x^2+axy+by^2$ with $a\in\fa\inv$ and $b\in\fa^{-2}$ up to the action of a certain group. Denote the set of all such forms by $W_\fa$, and for any ring $A$, let $W(A)$ denote the $A$-module of forms $x^2+rxy+sy^2$ for $r,s\in A$. The discriminant of a form $x^2+axy+by^2$ is $a^2-4b$.

	Let $G \subset \on{GL}_2$ denote algebraic subgroup of lower-triangular matrices of the form \[\bmat{1&\\\ast&*},\] and recall that there is a natural action of $G$ on $W$ given by \[g\cdot f(x,y)=f((x,y)g)\] for $g\in G$ and $f\in W$. The action of $G$ is given explicitly in coordinates by 
	\begin{equation}\label{eqn: N action coordinates}
		\bmat{1&\\u&\lambda}\cdot(x^2+axy+by^2)=x^2+(a\lambda+2u)xy+(b\lambda^2+ua\lambda+u^2)y^2.
	\end{equation} 
	Let $G_\fa$ denote the lattice in $G(K)$ of matrices with lower-triangular entry in $\fa\inv$ and diagonal entry in $R^\times$, and note that $G_\fa$ preserves the lattice $W_\fa$ in $W(K)$. Note that for $g\in G$ and $f\in W$, we have \[\disc(g\cdot f)=\det(g)^2\disc(f).\]
	
	Having established which objects are in play, we define a based map \[\Psi_\fa\colon \{\textrm{quadratic rings over }R\textrm{ with Steinitz class $[\fa]$}\} \to G_\fa\backslash W_\fa\] as follows. To each ring $S=R+\fa\xi$ with $\xi^2+a\xi+b=0$, let $\Psi_\fa(S)$ be the binary quadratic form $f(x,y)=x^2+axy+by^2$. A computation verifies that changing $\xi$ to $\lambda\xi+u$ for $\lambda\in R^\times$ and $u\in\fa\inv$ corresponds to changing $f$ by the element \[\bmat{1&\\u&\lambda}\in G_\fa,\] so the map is well-defined. 
	
	Conversely, given a form $f(x,y)=x^2+axy+by^2\in W_\fa$, consider the ring \[S=R[\fa\xi]/(\fa^2(\xi^2+a\xi+b)).\] Note that this defines a map \[W_\fa\to\{\textrm{quadratic rings over }R\textrm{ of Steinitz class $[\fa]$}\},\] and that this map sends forms belonging to the same $G_\fa$-equivalence class to the same ring up to isomorphism. The map is observably inverse to $\Psi_\fa$, implying the following theorem.
	\begin{theorem}[{\cite[Theorem 3.3]{ODorney}}]\label{thm: quad ring param}
		For each fractional ideal $\fa$ of $R$, the map $\Psi_\fa$ descends to a bijection \[\{\textrm{quadratic rings over }R\textrm{ with Steinitz class $[\fa]$}\}/\sim \; \to G_\fa\backslash W_\fa.\] Under this bijection, rings that are integral domains correspond precisely to the irreducible orbits, i.e., $G_\fa$-orbits of forms in $W_\fa$ that are irreducible over $K$. Moreover, the discriminant of a binary quadratic form multiplied by $\fa^2$ is the discriminant ideal of the corresponding quadratic ring. For any $f\in W_\fa$, we have that \[\Aut(\Psi_\fa\inv(f))\simeq\Stab_{G_\fa}(f)\simeq\Z/2\Z.\]
	\end{theorem}
	\begin{proof}
		We have already proven everything except the last statement. (The correspondence between irreducible forms and integral domains is immediate from the construction of $\Psi_\fa$ and $\Psi_\fa\inv$.) The automorphism group of any quadratic ring is isomorphic to $\Z/2\Z$. If $g=(u,\lambda)\in G_\fa$ stabilizes a form $f(x,y)=x^2+axy+by^2$, then $\lambda^2=\det(g)^2=1$ by considering discriminants. It follows that $\lambda=1$ and $u=0$ or $\lambda=-1$ and $u=a$. A computation shows that $g=(a,-1)$ stabilizes $f$, from which it follows that $\Stab_{G_\fa}(f)\simeq\Z/2\Z$. 
	\end{proof}
	We can completely describe the orbits of the $G$-action on $W$ over archimedean local fields. Compare the following with Lemma~\ref{lem: SL2 transitive action on V}. 
	\begin{lemma}\label{lem: G acts transitively on forms}
		The group $G(\C)$ acts transitively on the set of forms in $W(\C)$ with nonzero discriminant. The group $G(\R)$ acts transitively on the subset of forms in $W(\R)$ with discriminant having a given sign.
	\end{lemma}
	\begin{proof}
		Inspired by \eqref{eqn: N action coordinates}, given a form $f(x,y)=x^2+axy+by^2$, we set $2u=-a\lambda$ and act by $g=(-a\lambda/2,\lambda)$. A computation shows that \[g\cdot f(x,y)=x^2-\lambda^2(a^2-4b)y^2/4.\] Over $\C$, the result follows from choosing $\lambda$ to be some appropriate square root of $(a^2-4b)\inv$ that every form $f$ is $G(\C)$-equivalent to $x^2-y^2/4$. Over $\R$, every form of positive discriminant is $G(\R)$-equivalent to $x^2-y^2/4$; every form of negative discriminant is $G(\R)$-equivalent to $x^2+y^2/4$.
	\end{proof}
	
	\subsection{Gauss composition over a Dedekind domain}\label{subsec: Gauss composition}

	We now present the main result of Gauss composition over a Dedekind domain. Let $I=\fb_1\eta_1+\fb_2\eta_2$ be an ideal of $S=R+\fa\xi$, and suppose that \[\xi\eta_1=a\eta_1+c\eta_2\quad\textrm{and}\quad\xi\eta_2=b\eta_1+d\eta_2,\] where $a,d\in\fa\inv$, $b\in\fa\inv\fb_1\fb_2\inv$, and $c\in\fa\inv\fb_1\inv\fb_2$. We can construct a linear binary quadratic form from this data by considering \[f(x,y)=f(x\eta_1+y\eta_2)=cx^2+(d-a)xy-by^2.\] We may regard $f$ as a linear binary quadratic form in $\Sym^2(I^*)\otimes(\fa\inv\exterior^2(I))$ by considering $f(x,y)\eta_1\wedge\eta_2\in\fa\inv\exterior^2(I)$ as a tensor in $\Sym^2(I^*)\otimes\fa\inv\exterior^2(I)$. 
	
	We can free this map from its dependence on the basis of $I$ as follows. Begin by noting that $\fa\xi\subset R$, so $\xi\in\fa\inv$. Consider the map $I\to\fa\inv\exterior^2(I)$ taking $\omega\mapsto\omega\wedge\xi\omega$. We claim that this map is exactly the linear binary quadratic form $f$ found in the above. Indeed, if $\omega=x\eta_1+y\eta_2$, then \[\begin{split}
		f(\omega)&=(x\eta_1+y\eta_2)\wedge(ax\eta_1+cx\eta_2+by\eta_1+dy\eta_2)\\
		&=(cx^2+(d-a)xy-by^2)\eta_1\wedge\eta_2=f(x,y)\eta_1\wedge\eta_2.
	\end{split}\] We can see that $f$ is nonzero by tensoring everything with $K$ and noting that $f$ takes 1 to $1\wedge\xi\neq0$. Thus, we have a map taking a quadratic ring and ideal $(S,I)$ to a linear binary quadratic form in $\Sym^2(I^*)\otimes\fa\inv\exterior^2(I)$.
	
	Note that this map only depends on the ideal class of $I$. Let $I'=\kappa I$ for some $\kappa\in L$. If $I=\fb_1\eta_1+\fb_2\eta_2$, then $I'=\fb_1(\kappa\eta_1)+\fb_2(\kappa\eta_2)$, so the $\xi$-action on $I'$ is exactly the same as its action on $I$. The following tells us that this construction is bijective, discriminant-preserving, and that the invertible ideals correspond precisely to primitive forms. 
	\begin{theorem}[cf.\ {\cite[Theorem~4.2]{ODorney}} and {\cite[Corollary~4.2]{WoodGaussComp}}]\label{Gauss composition over a DD}
		Fix an ideal $\fa$ of $R$. The construction presented above gives a bijection between pairs $(S,I)$, where $S$ is an $\fa$-oriented quadratic ring and $I$ is a fractional ideal of $S$, up to equivalence, and isomorphism classes of pairs $(M,f)$, where $M$ is an $R$-lattice of rank $2$ and $f\in\Sym^2(M^*)\otimes\fa\inv\exterior^2(M)$ is a nonzero linear binary quadratic form. Two pairs $(S,I)$ and $(S',I')$ are said to be equivalent if there is an isomorphism $\phi\colon S\to S'$ (of oriented rings) such that $\phi(I)=\kappa I'$ for some $\kappa\in L$; two pairs $(M,f)$ and $(M',f')$ are said to be isomorphic if there is an isomorphism $\psi\colon M\to M'$ taking $f\mapsto f'$. Under this bijection, if $f$ is a form corresponding to some ring $S$, then the discriminant of $f$ is equal to the discriminant of $S$. Moreover, the invertible ideal classes correspond precisely to primitive forms.
	\end{theorem}

	\section{The parametrization}\label{subsec: param}
	
	The purpose of this section is to prove Theorem~\ref{thm:mainresult} and several of its corollaries. In \S\ref{subsec: BCFs and Bhargava Cubes}, we introduce the space of binary cubic forms and its interpretation in terms of Bhargava cubes. We then prove our parametrization in \S\ref{subsec: proof}, and in \S\ref{subsec: reducible param}, we explain how the (ir)reducibility of a binary cubic forms can be  detected in terms of the corresponding ideal class. In \S\ref{subsec: proj}, we conclude the section with a discussion of projective forms.
	
	\subsection{Binary cubic forms}\label{subsec: BCFs and Bhargava Cubes} 
	Henceforth, we impose the additional assumption that $R$ is a Dedekind domain with $\on{char}(R)\neq3$. As in \S\ref{sec-bqfs}, we first define binary cubic forms equivariantly and then transition to using more explicit objects. 
	\begin{definition}\label{def: BCF}
		A \emph{triply symmetric binary cubic form} over $R$ is the data of a locally free $R$-module $M$ of rank 2 and an element $f\in \Sym_3(M^*)$. Two triply symmetric binary cubic forms $(M,f)$ and $(M',f')$ are said to be \emph{isomorphic} if there exists an isomorphism $M\to M'$ taking $f\mapsto f'$. A \emph{linear triply symmetric binary cubic form} is the data of a locally free $R$-module $M$ of rank 2, a locally free $R$-module $L$ of rank 1, and an element $f\in\Sym_3(M^*)\otimes L$. An isomorphism of linear triply symmetric binary cubic forms is defined exactly as in Definition~\ref{def: BQF}.
	\end{definition}
	
	Now, fix an ideal $\fa$ of $R$. Going forward, we will be primarily concerned with linear triply symmetric binary cubic forms $(M,L,f)$ with $\exterior^2(M)\simeq\fa\simeq L$. Call these forms linear triply symmetric binary cubic forms \emph{of type} $\fa$, and note that any form of type $\fa$ is isomorphic to a form with $M=R\oplus\fa$. Henceforth, we will omit the data of $M$ and $L$ when referring to forms of type $\fa$. When no confusion is possible, we will often omit the adjectives ``triply symmetric'' and ``linear'' for the sake of brevity. Writing $M=R\alpha+\fa\beta$ for $\alpha,\beta\in M\otimes_RK$ and $L=\exterior^2(M)=\fa\alpha\wedge\beta$, a linear triply symmetric binary cubic form can be expressed as \[f(x,y)=f(x\alpha+y\beta)=(ax^3+3bx^2y+3cxy^2+dy^3)\alpha\wedge\beta,\] where $a\in\fa$, $b\in R$, $c\in\fa\inv$, and $d\in\fa^{-2}$. 
	
	For an arbitrary ring $A$, let $V(A)$ denote the set of triply symmetric binary cubic forms with coefficients in $A$. Given a fractional ideal $\fa$ of $R$, let $V_\fa$ denote the set of triply symmetric binary cubic forms of type $\fa$, which we note forms a subgroup of $V(K)$.
	
	There is a natural $\GL(M)$-action on the set of binary cubic forms of type $\fa$ given by \[g\cdot f(\zeta)=\det(g)\inv f(\zeta g)\] for $\zeta\in M$ and $g\in\GL(M)$. The resulting form $g\cdot f$ is clearly isomorphic to $f$; it is straightforward to verify that any form $f'$ of type $\fa$ isomorphic to $f$ is in the $\GL(M)$-orbit of $f$. With respect to our choice of basis $M=R\alpha+\fa\beta$, we see that \[\GL(M)\simeq\GL(R\oplus\fa)=\left\{\bmat{r & b\\a&s}\;\Big|\;r,s\in R,a\in\fa\inv,b\in\fa,\textrm{ and }rs-ab\in R^\times \right\},\] since $\GL(R\oplus\fa)$ acts on $M$ on the right. Thus, rather than considering isomorphism classes of triply symmetric binary cubic forms of type $\fa$, in the sequel we will work exclusively with orbits of forms of type $\fa$ under the action of certain subgroups of $\GL(R\oplus\fa)$. Of singular importance is the subgroup of $\GL(R\oplus\fa)$ of matrices with determinant 1: \[\SL(R\oplus\fa)=\{g\in\GL(R\oplus\fa)\mid\det(g)=1\}.\]
	
	To any binary cubic form \[f(x,y)=ax^3+3bx^2y+3cxy^2+dy^3\] of type $\fa$, we may associate two essential covariants, the first of which is called the \emph{discriminant}. The \emph{$($reduced$)$ discriminant} of $f$ is the quartic polynomial in its coefficients given by  \[\disc(f)=-\frac{\mathrm{Disc}(f)}{27}=-3b^2c^2+4ac^3+4b^3d+a^2d^2-6abcd,\] which we note is an element of $\fa^{-2}$. For an element $g\in\GL(R\oplus\fa)$, it is straightforwardly verified that \[\disc(g\cdot f)=\disc(f).\] The other covariant associated to $f$ is a certain binary quadratic form called the \emph{Hessian covariant} of $f$. In coordinates, if \[f(x,y)=ax^3+3bx^2y+3cxy^2+dy^3,\] then corresponding Hessian covariant is given by \[H(x,y)=-\frac{1}{36}\begin{vmatrix}
		f_{xx} & f_{xy}\\ f_{yx} & f_{yy},
	\end{vmatrix}=(b^2-ac)x^2+(ad-bc)xy+(c^2-bd)y^2,\] where we note that $b^2-ac\in R$, $ad-bc\in \fa\inv$, and $c^2-bd\in\fa^{-2}$. In the language of Definition~\ref{def: BQF}, we can view $H$ as a linear binary cubic form $(M,L,H)$ with $\exterior^2(M)=\fa$ and $L=\fa\inv\otimes\exterior^2(M)=R$.
	
	Equivariantly, we can construct the Hessian by viewing it as a \emph{slicing} of the \emph{Bhargava cube} corresponding to $f$. Regard $f$ as a trilinear map $f\colon M^{\otimes 3}\to\fa$, and note that $f\in\Sym_3(M^*)$ implies that $f\colon M^{\otimes3}\to\fa$ is invariant under any permutation of pure tensors in $M^{\otimes 3}$. Choosing any factor of $M$ in $M^{\otimes 3}$, use $f$ to construct a map $M\to\Hom(M^{\otimes2},\fa)\simeq\Hom(M,\fa M^*)$. Applying the determinant map to $\Hom(M,\fa M^*)$ gives us an element of $\Hom(\exterior^2(M),\exterior^2(\fa M^*))\simeq\fa^2\cdot(\exterior^2(M^*))^{\otimes2}$. Because $\exterior^2(M)=\fa$, we see that $\fa^2\cdot(\exterior^2(M^*))^{\otimes2}=R$. In short, we have the following chain of maps \begin{center}
		\begin{tikzcd}
			M\arrow[r,"f"] & \Hom(M,\fa M^*)\arrow[r,"\det"] & \Hom(\exterior^2(M),\fa^2\exterior^2(M^*))\arrow[r,"\sim"] & \fa^2\cdot(\exterior^2(M))^{\otimes 2}=R.
		\end{tikzcd}
	\end{center} In coordinates, we see that the composition of the above maps is simply the Hessian covariant $H$ of $f$.

	\subsection{Proof of Theorem~\ref{thm:mainresult}} \label{subsec: proof}
	
	In his work, O'Dorney showed that triples of ideals in quadratic ring extensions of $R$ are parametrized by Bhargava cubes whose entries lie in certain ideals of $R$ \cite{ODorney}. In the following, we will elicit a parametrization of 3-torsion in the relative class groups of quadratic ring extensions of $R$ by triply symmetric binary cubic forms (i.e., triply symmetric Bhargava cubes) whose entries lie in certain fractional ideals of $R$. 
	
	In this section, $S$ will always denote a quadratic $R$-algebra. Since we are taking $R$ and $K$ to be fixed quantities, our concern is with the \emph{relative} 3-torsion ideals of orders in quadratic extensions of $K$. This is particularly convenient, because $\Cl(S)[3]$ can be written as the direct sum $\Cl(S/R)[3]\oplus\Cl(R)[3]$ (see \S\ref{subsec: class group}), so understanding the relative $3$-torsion in $S$ tells us exactly the structure of the $3$-torsion subgroup of $\Cl(S)$. Thus, we will restrict our focus to the fractional ideals $I$ of $S$ whose ideal norm $[S:I]_R$ is principal, for whih Lemma~\ref{lem: rel tors Steinitz} tells us that $S$ and $I$ have the same Steinitz class. It therefore suffices to exhibit for each fractional ideal $\fa$ of $R$ a parametrization of the relative $3$-torsion ideal classes in $\fa$-oriented quadratic $R$-algebras. To this end, we let $\fa$ be a fixed fractional ideal of $R$ throughout the rest of this section.

	\begin{definition}(cf.\ \cite[Theorem~12]{BES})\label{def: balanced}
		Given an ideal $\fa$ of $R$, a quadruple $(S,I,\delta,s)$ consisting of 
		\begin{enumerate}
			\item a nondegenerate $\fa$-oriented quadratic $R$-algebra $S$ with orientation $\pi\colon \exterior^2(S)\to\fa$;
			\item a fractional ideal $I$ of $S$;
			\item an element $\delta\in (S\otimes_RK)^\times$;
			\item an element $s\in K^\times$;
		\end{enumerate} 
		is said to be $\fa$-\emph{balanced} if $I^3\subset\delta S$, $[S:I]_R=sR$, and $s^3=N(\delta)$. When $\fa$ is understood we repress it from our notation and simply refer to $(S,I,\delta,s)$ as \emph{balanced}. Note that the requirement that $[S:I]_R=sR$ implies that $\St(S)=[\fa]=\St(I)$ by Lemma~\ref{lem: rel tors Steinitz}. Two $\fa$-balanced quadruples $(S,I,\delta,s)$ and $(S',I',\delta',s')$ are said to be \emph{equivalent} if there exists a ring isomorphism $\phi\colon S\to S'$ and $\kappa\in S'\otimes_RK$ such that $I'=\kappa \phi(I)$, $\delta'=\kappa^3\phi(\delta)$, and $s'=N(\kappa)\phi(s)$. We denote this equivalence relation by $(S,I,\delta,s)\sim (S',I',\delta',s')$.
	\end{definition}
	
	For an $\fa$-balanced quadruple $(S,I,\delta,s)$, we remark that $s$ is simply a choice of generator of the principal ideal $[S:I]_R=sR$. Equivalently, by Lemma~\ref{ideal norm and steinitz index agree}, we may think of $s$ as an additional piece of data that restricts the possible choices of Steinitz decompositions of the ideal $I$ in the quadruple. In particular, when working with an $\fa$-balanced quadruple $(S,I,\delta,s)$, we will always choose a Steinitz decomposition $I=R\alpha+\fa\beta$ such that $\pi_K(\alpha\wedge\beta)=s$.

	On one side of our parametrization we have $\fa$-balanced quadruples. On the other side of the bijection we have orbits of binary cubic forms of type $\fa$. Recall that a binary cubic form of type $\fa$ is of the form \[ax^3+3bx^2y+3cxy^2+dy^3,\] where $a\in\fa$, $b\in R$, $c\in\fa\inv$, and $d\in\fa^{-2}$. Denote the set of all such forms by $V_\fa$. Recall from \S\ref{subsec: BCFs and Bhargava Cubes} that there is a natural $\GL(R\oplus\fa)$-action on $V_\fa$ given by \[g\cdot f(x,y)=\det(g)\inv f((x,y)g).\]

	Since, ultimately, we are interested in invertible ideals of $S$, we need to specify the forms that correspond to invertible ideals. Recall from Theorem~\ref{Gauss composition over a DD} the philosophy that quadratic rings over $R$ and ideal classes in those rings correspond to certain orbits of binary quadratic forms, and that invertible ideals give primitive forms. Moreover, recall that to each binary cubic form $f(x,y)=ax^3+3bx^2y+3cxy^2+dy^3\in V_\fa$ we may associate its Hessian covariant, $H(x,y)=(b^2-ac)x^2+(ad-bc)xy+(c^2-bd)y^2.$ Note that $b^2-ac\in R$, $ad-bc\in\fa\inv$, and $c^2-bd\in\fa^{-2}$. 
	\begin{definition}\label{def: projective}
		A binary cubic form in $V_\fa$ is said to be \emph{projective} if its Hessian covariant is primitive. In other words, $f(x,y)=ax^3+3bx^2y+3cxy^2+dy^3$ is not projective if and only if there exists a prime $\fp$ of $R$ dividing the ideals $(b^2-ac)R$, $(ad-bc)\fa$, and $(c^2-bd)\fa^2$.
	\end{definition}
	In our parametrization, we will show that the invertible ideals correspond precisely to projective binary cubic forms.

	Having introduced the objects on both sides of the parametrization, we now move to construct the map. We begin by offering a based map, where we choose a Steinitz decomposition of $I$ and illustrate the correspondence explicitly in this setting. We then explain how to construct this map equivariantly, without reference to an $R$-module basis of $I$. Let $S$ be an $\fa$-oriented quadratic $R$-algebra, where $\pi\colon \exterior^2 (S)\to\fa$ is the $\fa$-orientation. Write $S=R+\fa\xi$, and assume that $\pi(r\wedge a\xi)=ra$. Let $I$ be a fractional ideal of $S$, let $\delta$ be an element of $(S\otimes_RK)^\times$, let $s\in K^\times$, and assume that $(S,I,\delta,s)$ is $\fa$-balanced. Because $I$ has Steinitz class $[\fa]$, recall (from Remark~\ref{rmk: basis for Steinitz decomposition}) that there exist $\alpha,\beta\in S\otimes_RK$ such that $I=R\alpha+\fa\beta$. We may further assume, without loss of generality, that $\pi_K(\alpha\wedge\beta)=s$ (otherwise we can replace $\alpha$ or $\beta$ by a sufficient $K^\times$-multiple). Note that the element ideal norm $N(I;\pi,\alpha,\beta)$ is simply $s$. With this in mind, for each such $S$ we define a map \[\{(I,\delta,s)\mid(S,I,\delta,s)\textrm{ balanced}\}/\sim\;\longrightarrow\SL(R\oplus\fa)\backslash V_{\fa}\] as follows. Since $I^3\subset\delta S$, we may write
	\begin{equation}\label{system of equations}
		\begin{split}
			\alpha^3&=\delta(c_0+a_0\xi)\\
			\alpha^2\beta&=\delta(c_1+a_1\xi)\\
			\alpha\beta^2&=\delta(c_2+a_2\xi)\\
			\beta^3&=\delta(c_3+a_3\xi),
		\end{split}
	\end{equation} where $a_0\in\fa$, $a_1\in R$, $a_2\in\fa\inv$, and $a_3\in\fa^{-2}$. Let \[C(x,y)=a_0x^3+3a_1x^2y+3a_2xy^2+a_3y^3\] be the corresponding triply symmetric binary cubic form in $V_{\fa}$. This construction works for any $\fa$-oriented quadratic ring $S$, giving us a map from the set of equivalence classes of balanced quadruples $(S,I,\delta,s)$ to $V_\fa$. Call this map $\Phi_\fa$. 
	
	While this definition of $\Phi_\fa$ elucidates concretely the objects in play, its main drawback is its basis dependence---note that the above construction implicitly depends on our choice of bases for both $I$ and $S$. Henceforth, we opt for a more equivariant construction, inspired by \cite[Theorem~5.3]{ODorney} and \cite[Theorem~13]{BhargavaHCLI}. Using the orientation isomorphism $\pi\colon \exterior^2(S)\to\fa$, it is readily verified that \[C(x,y)=\pi\left(1\wedge\frac{(\alpha x+\beta y)^3}{\delta}\right),\] where $x$ and $y$ are variables taking values in $R$ and $\fa$, respectively. Thus, we may describe $C$ equivariantly as the map $I\to\fa$ taking $\zeta\mapsto\pi(1\wedge\zeta^3\delta\inv)$. 
	
	From this perspective, we see that changing $I=R\alpha+\fa\beta$ by any element $g$ of $\SL(R\oplus\fa)$ simply changes $C(x,y)$ (via the natural $\GL(R\oplus\fa)$-action on $V_\fa$) by that same element $g$. For the other direction, we see easily that any form in the $\SL(R\oplus\fa)$-equivalence class of $C(x,y)$ can be obtained from the data of $(S,I,\delta)$ simply by changing $I$ by an element of $\SL(R\oplus\fa)$ as needed. Therefore, $\Phi_\fa$ is well-defined. We will prove that $\Phi_\fa$ is a bijection.
	
	\begin{theorem}\label{thm: the parametrization}
		For each ideal $\fa$ of $R$, the map \[\Phi_\fa\colon \{\textrm{$\fa$-balanced quadruples } (S,I,\delta,s)\}/\sim\,\longrightarrow\SL(R\oplus\fa)\backslash V_{\fa}\] is a bijection. Moreover, $\Phi_\fa$ is discriminant-preserving, and under $\Phi_\fa$, the quadruples $(S,I,\delta,s)$ with $I$ invertible correspond precisely to projective forms. 
	\end{theorem}
	\begin{proof}
		In the above we illustrated how to associate to an $\fa$-balanced quadruple $(S,I,\delta,s)$ an element of $\SL(R\oplus\fa)\backslash V_\fa$. It remains to show that this mapping is a bijection: for a given binary cubic form $C(x,y)$ that there is \emph{exactly one} balanced quadruple $(S,I,\delta,s)$ (up to equivalence) that yields the ($\SL(R\oplus\fa)$-orbit of) $C$ under $\Phi_\fa$. Precisely, we will show how to reconstruct the ring $S$, orientation map $\pi$, ideal $I$, and elements $\delta$ and $s$ from the binary cubic form and then argue that $\Phi_\fa$ is injective.
		
		Philosophically, we have seen in Theorem~\ref{Gauss composition over a DD} that the data of an oriented quadratic ring over $R$ and an ideal class of a quadratic ring should correspond uniquely to some linear binary quadratic form. Thus, starting from a triply symmetric binary cubic form \[C(x,y)=a_0x^3+3a_1x^2y+3a_2xy^2+a_3y^3\in V_\fa,\] we will construct a linear binary quadratic form from which we can construct a ring $S$ and ideal $I$ of $S$. Then we will show that $I$ is balanced for some $\delta$ and $s$. 
		
		As in \S\ref{subsec: BCFs and Bhargava Cubes} of this section, regard $C(x,y)$ as a linear triply symmetric binary cubic form with $M=R\oplus\fa$ and $L=\exterior^2(M)=\fa$. Without loss of generality, we may assume that $a_1$ and $a_2$ are not both 0; otherwise, for $M=R\eta_1+\fa\eta_2$, we can change coordinates by $\eta_2\mapsto\eta_2+t\eta_1$ for some nonzero $t\in\fa\inv$. Recall that there is a natural linear binary quadratic form in $\Sym^2(R\oplus\fa\inv)\otimes\fa$ associated to $C$ called the \emph{Hessian}, which is given in coordinates by \[H(x,y)=(a_1^2-a_0a_2)x^2+(a_1a_2-a_0a_3)xy+(a_2^2-a_1a_3)y^2.\] Note that $a_1^2-a_0a_2\in R$, $a_1a_2-a_0a_3\in\fa\inv$, and $a_2^2-a_1a_3\in\fa^{-2}$. For simplicity, set $p=a_1^2-a_0a_2$, $q=a_1a_2-a_0a_3$, and $r=a_2^2-a_1a_3$. We proceed as in the proof of Theorem~\ref{Gauss composition over a DD} and build a ring $S$ and fractional ideal $I$ of $S$ from the Hessian $H$. Let $S$ be the ring $R[\fa\xi]/(\fa^2(\xi^2-q\xi+pr))$, oriented by $\pi:\exterior^2(S)\to\fa$, $1\wedge\xi\mapsto1$.

		Set \begin{equation}\label{ci definitions}
			\begin{split}
				c_0=a_1p-a_0q;\quad	c_1=-a_0r;\quad c_2=-a_1r;\quad c_3=-a_2r+a_3q. 	
			\end{split}
		\end{equation} Inspired by \eqref{system of equations}, let  \begin{equation}\label{alpha:beta}
			\frac{\alpha}{\beta}=\frac{c_1+a_1\xi}{c_2+a_2\xi},
		\end{equation} which uniquely determines $\alpha$ and $\beta$ up to some scalar factor in $S\otimes_R K$. Fixing a choice of $\alpha$ and $\beta$---for example, setting $\alpha=c_1+a_1\xi$ and $\beta=c_2+a_2\xi$---gives us an $R$-module $I=R\alpha+\fa\beta$. Using \eqref{alpha:beta}, a short calculation shows that \begin{equation}\label{eqn: xi action}
			\xi\alpha=p\beta\quad\textrm{and}\quad\xi\beta=-r\alpha+q\beta,
		\end{equation} from which it follows that $I$ is a sub-$S$-module of $S\otimes_RK$. Showing that $I\otimes_RK=S\otimes_RK$ is a bit subtle and is why we made the stipulation that not both $a_1$ and $a_2$ are 0. To see why $\alpha$ and $\beta$ are independent over $K$, suppose that $\alpha=c\beta$ for some nonzero $c\in K$. Then $\xi\alpha=c\xi\beta=p\beta$, implying that $\beta=0$ or $c\xi=p\in R$. By construction $\xi\not\in K$, so we must have $\alpha=\beta=0$, i.e., $a_1=a_2=c_1=c_2=0$. This is a contradiction. Thus, $I$ is indeed a fractional ideal of $S$. Setting \[\delta=\frac{\alpha^3}{c_0+a_0\xi}=\frac{\alpha^2\beta}{c_1+a_1\xi}=\frac{\alpha\beta^2}{c_2+a_2\xi}=\frac{\beta^3}{c_3+a_3\xi},\] we have that $I^3\subset\delta S$. From this it is clear that changing both $\alpha$ and $\beta$ by $\kappa\in S\otimes_R K$ changes $\delta$ by $\kappa^3$. Finally, set $s=\pi_K(\alpha\wedge\beta)$. In other words, let $s$ be the determinant of the change-of-basis matrix $g$ taking $1,\xi$ to $\alpha,\beta$. We see that changing $\alpha$ and $\beta$ by $\kappa\in S\otimes_RK$ changes $s$ by $N(\kappa)$. To prove that $(S,I,\delta,s)$ is balanced, it remains to show that $s^3=N(\delta)$, which follows from computing $s=\det(g)$ and $N(\delta)$ explicitly.
		
		Thus, given a binary cubic form $C(x,y)$, we have demonstrated how to construct a quadratic ring extension $S$ of $R$, an ideal $I$ of $S$, and elements $\delta\in S\otimes_R K$ and $s\in K^\times$ such that the quadruple $(S,I,\delta,s)$ is balanced. Starting from a form $C(x,y)$, if $(S,I,\delta,s)$ denotes the corresponding balanced quadruple, we see immediately that $\Phi_\fa(S,I,\delta,s)=C(x,y)$, so the map $\Phi_\fa$ is surjective. It remains to show injectivity and that two $\SL(R\oplus\fa)$-equivalent forms yield equivalent quadruples.
		
		Suppose $(S,I,\delta,s)$ and $(S',I',\delta',s')$ are two balanced quadruples whose images under $\Phi_\fa$ yield $\SL(R\oplus\fa)$-equivalent forms $C(x,y)$ and $C'(x,y)=g\cdot C(x,y)$, respectively. By replacing $I$ by $g\inv I$, we may assume without loss of generality that $C=C'$. It follows that the Hessians $H$ and $H'$ associated to $C$ and $C'$ are the same, and proceeding as in the proof of Theorem~\ref{Gauss composition over a DD} tells us that there is an isomorphism $\phi:S\to S'$ and $\kappa\in S'\otimes_RK$ such that $I'=\kappa\phi(I)$. We can replace $S'$ by $S$ and assume that $I'=\kappa I$. For all $\zeta\in I$, we have \[0=\pi(1\wedge\zeta^3\delta\inv)-\pi(1\wedge(\kappa\zeta)^3(\delta')\inv)=\pi(1\wedge(\delta\inv-\kappa^3(\delta')\inv)\zeta^3),\] so $(\delta\inv-\kappa^3(\delta')\inv)\zeta^3\in K$ for every $\zeta\in I$. We also have that $(\delta\inv-\kappa^3(\delta')\inv)x\in K$ for any $x\in I^3\subset\delta S$. Extending scalars to $K$, we see that $(\delta\inv-\kappa^3(\delta')\inv)x\delta\in K$ for every $x$ in $S\otimes_RK$. It follows that $(1-\kappa^3(\delta')\inv\delta)x\in K$ for all $x\in S\otimes_RK$, implying that $1-\kappa^3(\delta')\inv\delta=0$. Thus, $\delta'=\kappa^3\delta$. It remains to show that $s'=N(\kappa)s$. We have that \[s'R=[S:I']_R=[S:\kappa I]_R=N(\kappa)[S:I]_R=N(\kappa)sR,\] so $s'=uN(\kappa)s$ for some unit $u\in R^\times$. The balanced condition implies that $(s')^3=N(\delta')=N(\kappa)^3N(\delta)$, so we have that $u^3N(\kappa)^3s^3=N(\kappa)^3N(\delta)=N(\kappa)^3s^3$. Thus, $u^3=1$, and since $I'=u^2\kappa I$, $\delta'=(u^2\kappa)^3\delta$, and $s'=N(u^2\kappa)s$, we may conclude that $(S',I',\delta',s')\sim(S,I,\delta,s)$. 
		
		To see that the map is discriminant preserving, we simply recall that the discriminant of a binary cubic form is equal to the discriminant of its Hessian covariant. Crucially, this holds because our forms are triply symmetric! Since our construction of the ring $S$ and ideal $I$ was essentially exactly the same as in the proof of Theorem~\ref{Gauss composition over a DD}, we have that \[\disc(C)=\disc(H)=\disc(S),\] so the $\Phi_\fa$ is discriminant-preserving. 
		
		The correspondence between balanced quadruples $(S,I,\delta,s)$ with $I$ invertible and projective binary cubic forms follows from the analogous correspondence between invertible ideals and primitive binary quadratic forms. For a proof of this fact, we refer the reader to the proofs of Theorem~\ref{Gauss composition over a DD} given by O'Dorney \cite[Theorem~4.2]{ODorney} and Wood \cite[Section~2.1]{WoodGaussComp}.
	\end{proof}

	Note that Theorems~\ref{thm: form ideal correspondence} and~\ref{thm: the parametrization} combine to give a parametrization of 3-torsion ideal classes in $\fa$-\emph{oriented} quadratic rings over $R$ by $\SL(R\oplus\fa)$-orbits of $V_\fa$. Unfortunately, counting oriented quadratic rings is not feasible when the group of units of $R$ is infinite, as the $\fa$-orientations are in bijection with $R^\times$. To fix this issue, we instead consider the $\GL(R\oplus\fa)$-orbits of $V_\fa$, which corresponds to considering 3-torsion ideal classes in \emph{unoriented} quadratic rings over $R$ of Steinitz class $[\fa]$. Recall that $\GL(R\oplus\fa)$ acts on $V_\fa$ by \[g\cdot f(x,y)=\det(g)\inv f((x,y)g).\] The determinant map tells us that $\GL(R\oplus\fa)$ and $\SL(R\oplus\fa)$ fit into an exact sequence \[\begin{tikzcd}
		1\rar&\SL(R\oplus\fa)\rar&\GL(R\oplus\fa)\arrow[r,"\det"]& R^\times\rar&1.
	\end{tikzcd}\] This exact sequence splits via the map $R^\times\to \GL(R\oplus\fa)$ taking \[u\mapsto\bmat{u&\\&1}.\] Thus, to elicit an analogue of Theorem~\ref{thm: the parametrization} for $\GL(R\oplus\fa)\backslash V_\fa$, we simply need to compute the orbit of an arbitrary $\fa$-balanced quadruple $(S,I,\delta,s)$ under the action of the image of $R^\times$ in $\GL(R\oplus\fa)$ under the splitting map given above.
	
	Let $(S,I,\delta,s)$ be an $\fa$-balanced quadruple, and suppose that its image under the map $\Phi_\fa$ is the triply symmetric binary cubic form \[f(x,y)=a_0x^3+3a_1x^2y+3a_2xy^2+a_3y^3\in V_\fa,\] so $a_0\in\fa$, $a_1\in R$, $a_2\in\fa\inv$, and $a_3\in\fa^{-2}$. Write $S=R+\fa\xi$. Let $u\in R^\times$ and consider the element \[\bmat{u&\\&1}\in\GL(R\oplus\fa),\] which, by abuse of notation, we also denote using $u$. Under the action of $\GL(R\oplus\fa)$, the forms $f$ and $uf$ are identified, where \[u\cdot f(x,y)=a_0u^2x^3+3a_1ux^2y+3a_2xy^2+a_3u\inv y^3.\] Let $(S',I',\delta',s')$ denote the $\fa$-balanced triple corresponding to $uf$, and set $\xi'=u\inv\xi$. Tracing back through the proof of Theorem~\ref{thm: the parametrization}, we see that $S'=R[\fa\xi']/(\fa^2((\xi')^2-q\xi'+pr))$ with orientation given by $\pi_K'(1\wedge\xi')=u\inv$. Then $\alpha'=u^2(c_1+a_1\xi')$ and $\beta'=u(c_2+a_2\xi')$, so $I'=R\alpha'+\fa\beta'$. The elements $\delta'$ and $s'$ are similarly constructed, with \[\delta'=\frac{u^3(c_1+a_1\xi')}{c_0+a_0\xi'}\] and $s'=\pi_K'(\alpha'\wedge\beta')=u^2(a_2c_1-a_1c_2)=u^2s$. We see that the quadruples $(S,I,\delta,s)$ and $(S',I',\delta',s')$ are \emph{almost} equivalent under $\sim$ by considering the map $\phi\vcentcolon S\to S'$ taking $\xi\mapsto \xi'$ and setting $\kappa=u$. However, the orientations on $S$ and $S'$ are inequivalent, implying that $(S,I,\delta,s)\not\sim(S',I',\delta',s')$. With this in mind, we consider the following stronger equivalence relation on the set of balanced quadruples.
	\begin{definition}
		We say two $\fa$-balanced quadruples $(S,I,\delta,s)$ and $(S',I',\delta',s')$ are \emph{generally equivalent} if there is an isomorphism $\phi\vcentcolon S\to S'$ of unoriented rings and $\kappa\in S'\otimes_RK$ such that $I'=\kappa\phi(I)$, $\delta'=\kappa^3\phi(\delta)$, and $s'=N(\kappa)\phi(s)$. We denote this equivalence relation by  $(S,I,\delta,s)\approx(S',I',\delta',s')$.
	\end{definition}
	
	With this definition in hand, we summarize the passage to $\GL(R\oplus\fa)$-orbits in the following theorem.
	\begin{theorem}\label{thm: GL parametrization}
		The map $\Phi_\fa$ described in Theorem~\ref{thm: the parametrization} induces a discriminant-preserving bijection \[\{\fa\textrm{-balanced quadruples }(S,I,\delta,s)\}/\approx\,\longrightarrow\GL(R\oplus\fa)\backslash V_\fa.\] Under this bijection, the quadruples $(S,I,\delta,s)$ with $I$ invertible correspond precisely to projective forms.
	\end{theorem}
	
	Given an $\fa$-balanced quadruple $(S,I,\delta,s)$, Lemmas~\ref{proper cont upstairs implies proper cont downstairs} and~\ref{ideal norm and steinitz index agree} imply that if $I$ is invertible, then $I^3=\delta S$: if $I$ is invertible and $I^3\subsetneq\delta S$, then the expression \[R=N(\delta)\inv s^3R=N(\delta)\inv[S:I]_R^3=[S:\delta\inv I^3]_R\subsetneq R\] is a contradiction, where the third equality in the above follows from Lemma~\ref{ideal norm agrees with element norm}. 
	
	Immediately, we see that the general equivalence classes of balanced quadruples $(S,I,\delta,s)$ with $I$ invertible correspond to relative 3-torsion ideal classes in $\Cl(S/R)[3]$. \begin{theorem}\label{thm: form ideal correspondence}
		Fix a fractional ideal $\fa$ of $R$. Let $S$ be a quadratic $R$-algebra equipped with any orientation. The map \[\{(S,I,\delta,s)\in\{S\}\times\cali(S)\times(S\otimes_RK)^\times\times K^\times\mid (S,I,\delta,s)\textrm{ is $\fa$-balanced}\}/\approx\;\longrightarrow\Cl(S/R)[3]\] given by $(S,I,\delta,s)\mapsto[I]$ is a surjective homomorphism whose kernel is isomorphic to $S_{N=1}^\times/S_{N=1}^{\times3}$, where $S_{N=1}^\times$ denotes the subgroup of units of $S$ with norm 1 and where $S_{N=1}^{\times3}=\{x^3\mid x\in S_{N=1}^\times\}$.\footnote{The natural analogue of this theorem for balanced quadruples under the equivalence relation $\sim$ also holds. In this case, the kernel is isomorphic to $(S_{N=1}^\times/S_{N=1}^{\times3})\times R^\times$, where the copy of $R^\times$ comes from taking into account the possible orientations on $S$.} 
	\end{theorem}
	\begin{proof}
		Firstly, we note that the map is well-defined because $I$ being invertible forces $I^3=\delta S$. Moreover, we already have that $I\in\Cl(S/R)$ (since $[S:I]_R$ is principal) so indeed $[I]\in\Cl(S/R)[3]$. 
		
		Given $[I]\in\Cl(S/R)[3]$, we have that $I^3=\delta S$ for some $\delta\in (S\otimes_RK)^\times$. Applying the ideal norm, we see that $[S:I]_R^3=N(\delta)R$. Let $\pi:\exterior^2(S)\to\fa$ be an orientation of $S$. Choosing any Steinitz decomposition $I=R\alpha+\fa\beta$ of $I$, set $s=\pi_K(\alpha\wedge\beta)$, and note that $[S:I]_R=sR$. Thus, we have that $s^3=uN(\delta)$ for some $u\in R^\times$, which we may rewrite as \[(su\inv)^3=N(\delta u\inv).\] It follows that $(S,I,\delta u\inv,su\inv)$ is a balanced quadruple corresponding whose image under the forgetful map is $[I]$. Any other representative $I'$ of the same ideal class $[I]$ produces a generally equivalent triple: if $I'=\kappa I$, then we see that $(S,I,\delta,s)\approx(S,I',\kappa^3\delta,N(\kappa)s)$. Therefore, the map is indeed surjective.
		
		Finally, suppose that $(S,I,\delta,s)$ is such that $[I]$ is the trivial class. Then $I=\kappa S$ for some $\kappa\in(S\otimes_RK)^\times$, and we have that $\kappa^3S=\delta S$ and $s^3=N(\delta)$. We see that $(S,S,\delta/\kappa^3,s/N(\kappa))\approx (S,I,\delta,s)$, where $\delta/\kappa^3\in S^\times$ and $s/N(\kappa)\in R^\times$. 
		
		Thus, we may count the general equivalence classes of balanced quadruples whose image in $\Cl(S/R)[3]$ is trivial by counting equivalence classes of pairs $(u,v)\in S^\times\times R^\times$ such that $v^3=N(u)$, where two pairs $(u,v)$ and $(u',v')$ are equivalent if there exists $\kappa\in S^\times$ such that $u'=\kappa^3u$ and $v'=N(\kappa)v$. Immediately, we see that each pair $(u,v)$ is equivalent to $(u',1)$ for some $u'\in S^\times$ by setting $\kappa=vu\inv$. The theorem follows from noting that two pairs $(u,1)$ and $(u',1)$ are equivalent if and only if there exists some $\kappa\in S_{N=1}^\times$ such that $u'=\kappa^3 u$.
	\end{proof}

	We conclude by demonstrating how to correct for the $|S_{N=1}^\times/S_{N=1}^{\times3}|$-fold overcount when $R$ is the ring of integers of a number field $K$. Doing so will imply Corollary~\ref{cor: orbit corr}. Let $K$ be a number field, and let $R$ denote its ring of integers. Suppose that $K$ has $r_1$ real places and $r_2$ complex places so that $[K:\Q]=r_1+2r_2$. Let $S$ be an order in a quadratic extension $L$ of $K$. Let $S_{N=1}^\times$ denote the subgroup of units of $S$ with norm 1, and let $S_{N=1}^{\times3}=\{x^3\mid x\in S_{N=1}^\times\}$. We first show that it suffices to compute $|S^\times/S^{\times3}|$.

	\begin{prop}
		\label{prop: units mod cubes in terms of other stuff}
		Let $S$ be an order in a quadratic extension $L$ of $K$. Then 
		\[ |S_{N=1}^\times/S_{N=1}^{\times3}|=\frac{|S^\times/S^{\times3}|}{|R^\times/R^{\times3}|} \]
	\end{prop}
	\begin{proof}
		Begin by considering the chain of maps 
		\begin{equation*}
			\begin{tikzcd}
				S_{N=1}^\times\rar & S^\times\arrow[r,"N"]& R^\times,
			\end{tikzcd}
		\end{equation*}
		and note that $S_{N=1}^\times$ is the kernel of the norm map. Taking quotients by the subgroups of cubes, we get 
		\begin{equation}\label{eqn: exact seq units}
			\begin{tikzcd}
				S_{N=1}^\times/S_{N=1}^{\times3}\rar & S^\times/S^{\times 3}\arrow[r,"N"]& R^\times/R^{\times3},
			\end{tikzcd}
		\end{equation} 
		where we see that the kernel of the norm map $S^\times/S^{\times 3}\to R^\times/ R^{\times3}$ is $S_{N=1}^\times/S_{N=1}^{\times3}$. We claim that the sequence \eqref{eqn: exact seq units} is exact. In particular, we define a splitting $R^\times/R^{\times3}\to S^\times/S^{\times 3}$ taking $u\mapsto u^2$. This map is clearly well-defined, and we see that for $u\in R^\times/ R^{\times3}$, we have $N(u^2)=u^4=u$ in $R^\times/ R^{\times3}$. Thus, the sequence is exact; the result follows.
	\end{proof}

	Next, we show that it suffices to compute $|S^\times/S^{\times3}|$ for maximal orders. Fix a quadratic extension $L$ of $K$; let $\cals$ denote its ring of integers. Suppose that $S$ is an order in $L$, and recall that $\cals$ is the integral closure of $S$ in $L$. Then we have the following generalization of Dirichlet's unit theorem for orders.
	\begin{proposition}[e.g., {\cite[Theorem~I.12.12]{Neukirch}}]
		With notation as in the preceding paragraph, we have that $\rank(S)=\rank(\cals)=r+s-1$, where $r$ and $s$ denote the number of real and complex places of $L$.
	\end{proposition}
	Thus, we have reduced the problem of computing $|S_{N=1}^\times/S_{N=1}^{\times3}|$ to one of understanding the unit group of $\cals$. Ultimately, we would like to do so in terms of the unit group of $R$. Recall that any quadratic extension of $L$ is isomorphic $K(\sqrt{D})$ for $D\in K^\times$ such that $x^2-D$ is irreducible. By Dirichlet's unit theorem, we know that $\rank(R^\times)=r_1+r_2-1$. The unit group of $\cals$ is then entirely determined by the ramification data for each of the infinite places of $K$ in $L$.
	
	Let $\nu$ be an archimedean place of $K$. Recall that 
	\begin{equation}\label{eqn: tensor split}
		L\otimes_K K_\nu=\prod_{\mu}K_\mu,
	\end{equation} 
	where the product runs over the archimedean absolute values $\mu$ on $L$ extending $\nu$. If $\nu$ is complex, then $\nu$ is inert, and the product on the right-hand side of \eqref{eqn: tensor split} is another copy of $\C$. If $\nu$ is real, then there are two possibilities for the decomposition of $\nu$ in $L$: either $\nu$ extends to a complex place (it ramifies) or it splits into two real places. If $s_1$ denotes the number of real places of $K$ that split and $s_2$ the number of real places that ramify, then we see easily that $L$ has $2s_1+s_2+r_2$ infinite places. Thus, the free part of $\cals^\times$ is isomorphic to $\Z^{2s_1+s_2+r_2-1}$. To obtain a formula for $|\cals^\times/\cals^{\times 3}|$, we will need some way of quantifying $s_1$ and $s_2$ in terms of data pertaining to the base field $K$. 
	
	Enter the theory of specifications. Let $K_\infty$ denote the product of the completions of $K$ at the infinite places, i.e., we have $K_\infty=\prod_{\fp|\infty}K_\fp.$ Note that $K_\infty\simeq\R^{r_1}\times\C^{r_2}$ as real vector spaces.
	\begin{definition}\label{def: specifications}
		Consider the group $\{\pm1\}^{r_1}\times\{1\}^{r_2}$, whose elements we henceforth refer to as \emph{specifications}. We define a surjective homomorphism $s \colon K_\infty^\times\to\{\pm1\}^{r_1}\times\{1\}^{r_2}$ as follows. Given an element $\Delta=(\Delta_\fp)_\fp\in K_\infty^\times$, let $s(\Delta)$ be the tuple in $\{\pm1\}^{r_1}\times\{1\}^{r_2}$ indexed by the archimedean places of $K$ with 1's at all the complex places and $\mathrm{sign}(\Delta_\fp)\cdot 1$ at the real places $\fp$. Note that $s$ has kernel $\R_{+}^{r_1}\times\C^{r_2}$. For a fixed specification $\sigma$, we say that $\Delta\in K_\infty^\times$ \emph{has specification} $\sigma$ if $s(\Delta)=\sigma$.
	\end{definition}
	
	Let $\Delta=4D$ be the discriminant of $L$, and recall that the isomorphism class of $L$ is determined entirely by $\Delta$. Note that $D$ and $\Delta$ are elements of $K$ and that they have the same specification. Suppose that $\Delta$ has specification $\sigma=(\sigma_i)\in\{\pm1\}^{r_1}\times\{1\}^{r_2}$. Call $\sigma$ the \emph{specification of} $L$. Let $t_1$ be the number of $+1$'s appearing in $\sigma$; let $t_2$ be the number of $-1$'s. We will show that the unit group of $\cals$ depends only on the specification of $\Delta$ and, in particular, that $s_1=t_1$ and $s_2=t_2$. 
	
	Let $\nu$ be a real place of $K$, and suppose that $\sigma_\nu$ is the coordinate of $\sigma$ corresponding to $\nu$. Suppose that $\sigma_\nu=+1$. The place $\nu$ gives an embedding $K\hookrightarrow K_\nu\simeq\R$, and under this embedding $\sigma_\nu=+1$ implies that $D$ is sent to an element of $\R_{+}$. Tensoring $K_\nu$ with $K(\sqrt{D})=L$, we see that $L\otimes K_\nu$ is isomorphic to $\R\otimes\R$. If $\sigma_\nu=-1$, then the embedding $K\hookrightarrow K_\nu\simeq\R$ induced by $\nu$ takes $D$ to $\R_{-}$. Again, tensoring $K_\nu$ with $K(\sqrt{D})=L$, we see that $L\otimes K_\nu$ is isomorphic to $\C$. It follows that $s_1$ and $s_2$ indeed are the number of $+1$'s and $-1$'s appearing in $\sigma$. We summarize these observations in the following proposition. Let $\varsigma_{-}=1$ and $\varsigma_+=3$.
	\begin{prop}\label{prop: units mod cubes max orders}
		Let $L$ be a quadratic extension of $K$ with ring of integers $\cals$ and specification $\sigma$. If $L$ does not contain cube roots of unity, then \[|\cals^\times/\cals^{\times3}|=3^{2s_1+s_2+r_2-1}=3^{r_1+r_2-1}\prod_{i=1}^{r_1}\varsigma_{\sigma_i}.\] If $L$ contains the cube roots of unity, then \[|\cals^\times/\cals^{\times3}|=3^{r_1+r_2}\prod_{i=1}^{r_1}\varsigma_{\sigma_i}.\]
	\end{prop}
	Combining Propositions~\ref{prop: units mod cubes in terms of other stuff} and~\ref{prop: units mod cubes max orders}, we have the following result. 
	\begin{corollary}\label{cor: norm 1 units mod cubes}
		Let $S$ be an order in a quadratic extension of $K$ with specification $\sigma$. If $S$ does not contain cube roots of unity, or if both $S$ and $R$ contain cube roots of unity, then \[|S_{N=1}^\times/S_{N=1}^{\times3}|=\prod_{i=1}^{r_1}\varsigma_{\sigma_i}.\] If $S$ contains cube roots of unity and $R$ does not, then \[|S_{N=1}^\times/S_{N=1}^{\times3}|=3\prod_{i=1}^{r_1}\varsigma_{\sigma_i}.\]
	\end{corollary}
	
	Corollary~\ref{cor: orbit corr} now follows from aggregating Theorem~\ref{thm: GL parametrization}, Theorem~\ref{thm: form ideal correspondence}, and Corollary~\ref{cor: norm 1 units mod cubes}.
	
	\subsection{Reducible forms}\label{subsec: reducible param}
	We have seen that under the correspondences of Theorem~\ref{thm: the parametrization} and Theorem~\ref{thm: form ideal correspondence}, the projective forms correspond to invertible 3-torsion ideal classes in $\Cl(S/R)[3]$.  Bhargava and Varma showed that the reducible forms correspond exactly to the 3-torsion ideals in the \emph{ideal group}. In the following we prove an analogue of the correspondence between reducible forms and 3-torsion ideals.
	
	\begin{lemma}\label{lem: red corr to cubes}
		Fix an ideal $\fa$ of $R$. Let $f\in V_\fa$, and let $(S,I,\delta,s)$ be a representative of $\Phi_\fa\inv(f)$. Then $f$ has a $K$-rational zero as a binary cubic form if and only if $\delta =\gamma^3 $ for some $\gamma\in (S\otimes_RK)^\times$.
	\end{lemma}
	\begin{proof}
		Suppose $\delta=\gamma^3$ for some $\gamma\in(S\otimes_RK)^\times$. Replace $I$ by $\gamma\inv I$, $\delta$ by $1$, and $s$ by $N(\gamma)\inv s$, yielding the equivalent balanced quadruple $(S,\gamma\inv I,1,N(\gamma)\inv s)$. Note that $(N(\gamma)\inv s)^3=1$. We remark that it is sufficient to show that any $\GL_2(K)$-translate of $f$ has a $K$-rational zero. As before, write $I=R\alpha+\fa\beta$ for some $\alpha,\beta\in S\otimes_RK$ with $\pi_K(\alpha\wedge\beta)=s$. Let $\lambda\in I\cap R$, and complete $\lambda$ to a $K$-basis for $(S\otimes_RK)^\times=K\lambda+K\mu$. Expressing $\alpha$ and $\beta$ in terms of $\lambda$ and $\mu$, say $\alpha=\ell_1\lambda+m_1\mu$ and $\beta=\ell_2\lambda+m_2\mu$ for $\ell_i,m_i\in K$ and $1\leq i\leq 2$, we see that \[I=(\ell_1R+\ell_2\fa)\lambda+(m_1R+\fa m_2)\mu.\] Then $f$ is $\GL_2(K)$-equivalent to the form given by \[\pi(1\wedge(x\lambda+y\mu)^3).\] Since $\lambda\in R$, we have $x^3\lambda^3\in K$ and $\pi(1\wedge x^3\lambda^3)=0$. It follows that some $\GL_2(K)$-translate of $f$ has vanishing $x^3$-coefficient and hence a $K$-rational zero. Therefore, $f$ must have a $K$-rational zero. 
		
		Conversely, suppose that $f(x_0,y_0)=0$ for some $x_0,y_0\in K$. Writing $I=R\alpha+\fa\beta$, we see that \eqref{system of equations} implies that \begin{equation}\label{eqn: delta cube}
			(x_0\alpha+y_0\beta)^3=\delta(c_0x_0^3+3c_1x_0^2y_0+3c_2x_0y_0^2+c_3y_0^3),
		\end{equation} where we recall that each $c_i$ lies in $\fa^{-i}$. Since $x_0,y_0\in K$, we have that $r=c_0x_0^3+3c_1x_0^2y_0+3c_2x_0y_0^2+c_3y_0^3\in K$. Now, set $\zeta=\alpha x_0+\beta y_0$, and apply norms to both sides of the equation above. This gives $N(\zeta)^3=r^2N(\delta)$. Since $(S,I,\delta,s)$ is balanced, it follows that $N(\delta)=s^3$, and we have that $r^2$ is a cube. It follows that $r$ itself is a cube, and \eqref{eqn: delta cube} tells us that $\delta=(x_0\alpha+y_0\beta)^3/r$.
	\end{proof}

	Let $V_\fa^\red$ denote the subset of reducible forms in $V_\fa$. For $S$ an oriented quadratic ring over $R$ of type $\fa$, let $\calh(S)^\red$ denote the subgroup of $\fa$-balanced quadruples $(S,I,\delta,s)$ for which $\delta$ is a cube; by Lemma~\ref{lem: red corr to cubes}, this is isomorphic under $\Phi_\fa$ to the subgroup of reducible forms corresponding to $S$. Define a map 
	\[\vphi_\fa(S) \colon \mathcal{I}(S/R)[3]\to \calh(S)^\red,\qquad I\mapsto(S,I,1,1).\] 
	Immediately, we see that $\im(\vphi_\fa(S))\subset\calh(S)^\red$. In fact, $\vphi_\fa(S)$ is an isomorphism $\mathcal{I}(S/R)[3]\to\calh(S)^\red$.  
	\begin{theorem}\label{thm: red corresponds to 3-torsion ideals}
		The map $\vphi_\fa$ is an isomorphism between $\mathcal{I}(S/R)[3]$ and $\calh(S)^\red$.
	\end{theorem}
	\begin{proof}   
		Suppose that the relative 3-torsion ideal $I$ in $\mathcal{I}(S/R)[3]$ is mapped to a quadruple equivalent to the identity element $(S,S,1,1)$. That is, suppose $(S,S,1,1)\sim (S,I,1,1)$. Then there exists $\kappa\in (S\otimes_RK)^\times$ such that $I=\kappa S$ and $1=\kappa^3$. It follows that $\kappa\in S$ and hence that $I=\kappa S=S$. To see why the map is surjective, let $(S,I,\delta,s)\in\calh(S)^\red$ so that $\delta=\gamma^3$ for some $\gamma\in (S\otimes_RK)^\times$. Then $(S,I,\delta,s)$ is equivalent to $(S,\gamma\inv I,1,N(\gamma)\inv s)$, where we note that $u=N(\gamma)\inv s\in R^\times$ is such that $u^3=1$. Noting that $(S,\gamma\inv I,1,u)\sim(S,u^{-2}\gamma\inv,1,1)=\vphi_\fa(S)(u^{-2}\gamma\inv I)$, the result follows.  
	\end{proof}
	
	\begin{corollary}
		If $S$ is a maximal order, then $\calh(S)^\red$ contains only the identity element $(S,S,1)$. 
	\end{corollary}
	\begin{proof}
		This follows because maximal orders are Dedekind domains, and $\mathcal{I}(S/R)[3]$ is trivial when $S$ is a Dedekind domain.
	\end{proof}
	
	\subsection{Projectivity}\label{subsec: proj} Recall from \S\ref{subsec: BCFs and Bhargava Cubes} that a projective binary cubic form $f\in V_\fa$ is a form for which certain $\fa$-multiples of the coefficients of its Hessian covariant do not share a common prime divisor. Precisely, if \[f(x,y)=ax^3+3bx^2y+3cxy^2+dy^3\] for $a\in\fa$, $b\in R$, $c\in\fa\inv$, and $d\in\fa^{-2}$, then $f$ is projective as long as the ideals \[(b^2-ac)R;\quad (ad-bc)\fa;\quad (c^2-bd)\fa^2\] do not share a common prime factor. 
	
	In this subsection, we prove several facts about the set of projective forms in $V_\fa$. In particular, we compute its $\fp$-adic volume for each prime $\fp$. To do so, we establish the following $\fp$-adic notation. Given a prime $\fp$ of $R$ and an element $x\in K$, let $\nu_\fp(x)$ denote the standard $\fp$-adic valuation of $x$, and let $|x|_\fp=N(\fp)^{-\nu_\fp(x)}$ denote the standard $\fp$-adic absolute value of $x$. Let $K_\fp$ denote the completion of $K$ at $\fp$. Given a set $X\subset V(K)$, let $X_\fp$ denote its $\fp$-adic closure in $V(K_\fp)$. 
	\begin{definition}
		For each prime $\fp$ of $R$, let $Z_\fp$ be a nonempty, open subset of $\{v\in V(R_\fp)\mid \disc(v)\neq0\}$ with the following significance: $Z_\fp$ has boundary of measure 0 and is preserved by $\GL(R\oplus\fa)_\fp$, which we take to be the closure of $\GL(R\oplus\fa)$ in $\GL_2(K_\fp)$. A set $Z\subset V_\fa$ is said to be \emph{defined by congruence conditions} if every element $v$ of $Z$, when regarded as an element of $V(K_\fp)$, is also an element of $Z_\fp$ for a collection of $Z_\fp$'s as defined above. A set $Z$ defined by congruence conditions is said to be \emph{acceptable} if $Z_\fp$ contains all forms $v\in V(R_\fp)$ such that $\nu_\fp(\disc(v))<2$ for all but finitely many primes $\fp$ of $R$. 
	\end{definition}

	\begin{lemma}\label{lem: projective is acceptable}
		The set $\calp$ of projective forms in $V_\fa$ is $\GL(R\oplus\fa)$-invariant, acceptable set defined by congruence conditions.
	\end{lemma}
	\begin{proof}
		Recall that a form $f(x,y)=ax^3+3bx^2y+3cxy^2+dy^3$ in $V_\fa$ is projective if and only if there is no prime $\fp$ dividing all three of \[(b^2-ac)R,\quad (ad-bc)\fa,\quad (c^2-bd)\fa^2,\] which are the coefficients of the Hessian covariant of $f$. The $\GL(R\oplus\fa)$-invariance of $\calp$ follows from the fact that acting by an element $g$ of $\GL(R\oplus\fa)$ on a primitive form $f\in S$ relates the Hessians of $f$ and $gf$ by $g$. Because the translation of an imprimitive binary quadratic form (that is, one whose coefficients are all divisible by some prime $\fp$) by an element of $\GL(R\oplus\fa)$ is also imprimitive, it follows that $gf$ is also projective. 
		
		Now, note that $\calp_\fp$ is simply the set of forms $f$ in $(V_\fa)_\fp$ such that the coefficients of the Hessian of $f$ are not all divisible by $\fp$. Emulating the argument in the previous paragraph shows that $\calp_\fp$ is preserved by $\GL(R\oplus\fa)_\fp$. Thus, $\calp$ is defined by congruence conditions. For all primes not dividing $\fa$, we claim that $\calv_\fp\subset\calp_\fp$. Let $v\in\calv_\fp$ so that $\disc(v)$ is not divisible by $\fp^2$. The discriminant of $v$ is the discriminant of its Hessian, from which we may conclude that if $v$ is not projective, then $\fp^2$ divides $\disc(v)$. Acceptability follows. 
	\end{proof}

	Regard $V(R_\fp)$ as $R_\fp^4$ and equip $V(R_\fp)$ with its Haar measure $\mu$, normalized so that the closure of $V_\fa$ in $V(R_\fp)$ has volume 1. We let $\mu(\calp)$ denote the $\fp$-adic density of $\calp_\fp$ in $V(R_\fp)$. Next, we compute the volume of $\calp_\fp$ in $V(R_\fp)$. 
	\begin{lemma}\label{lem: proj vol}
		We have $\mu(\calp)=1-N(\fp)^{-2}$. 
	\end{lemma} 
	\begin{proof}
		Recall that a form $f(x,y)=ax^3+3bx^2y+3cxy^2+dy^3$ in $V_\fa$ is projective if and only if there is no prime $\fp$ dividing all three of \[(b^2-ac)R,\quad (ad-bc)\fa,\quad (c^2-bd)\fa^2.\] Since $\calp$ is defined by congruence conditions modulo $\fp$, we can compute $\mu(\calp)$ by counting the number of projective forms modulo $\fp$. Let $\kappa$ denote the residue field $R/\fp$. We have two cases depending on whether $\fp|\fa$. 
		
		If $\fp$ does not divide $\fa$, then it suffices to count the number of solutions modulo $\fp$ to  \begin{equation}\label{eqn: proj eqn}
			b^2-ac=ad-bc=c^2-bd=0\mod\fp.
		\end{equation} Note that $a$ and $b$ can take any pair of values except $(0,m)$ for $m\neq0$ mod $\fp$. For any such nonzero $(a,b)$, the values of $c$ and $d$ are determined modulo $\fp$. If $(a,b)=(0,0)$, then $c$ is 0 modulo $\fp$ and $d$ can be anything (modulo $\fp$). Thus, there are $|\kappa|^2-(|\kappa|-1)+(|\kappa|-1)=|\kappa|^2$ solutions to \eqref{eqn: proj eqn}, which implies $\mu(\calp)=(|\kappa|^4-|\kappa|^2)/|\kappa|^4=1-N(\fp)^{-2}$. 
		
		If $\fp$ divides $\fa$, then rewrite $a=\varpi a'$, $c=c'/\varpi$, and $d=d'/\varpi^2$, where $\varpi$ is the uniformizer of $R_\fp$ and $a',$ $c',$ and $d'$ are in $R_\fp$. The projectivity condition implies that it suffices to count the number of solutions modulo $\fp$ to \[b^2-a'c'=a'd'-bc'=(c')^2-bd'=0\mod\fp.\] Emulating the arguments from the previous paragraph tells us that $\mu(\calp)=1-N(\fp)^{-2}$. 
	\end{proof}
	
	\section{Orbits over local fields and an ad\`elic reinterpretation}\label{sec: reduction theory}
	
	In \S\ref{subsec: reduction theory}, we analyze the $\GL_2$-action on $V$ over certain local fields. Understanding the orbit structure over local fields is both of independent algebraic interest and essential for applying the parametrization from \S\ref{subsec: param} to count 3-torsion in class groups of orders in quadratic extensions of number fields. We use several of the results from \S\ref{subsec: reduction theory} in \cite{HSAnalysis}. In \S\ref{sec: volumes}, we offer an interpretation of the group action of $\GL(R\oplus\fa)$ on $V_\fa$ from an ad\`elic perspective akin to the one used by Bhargava, Shankar, and Wang to count cubic extensions of global fields \cite{BSW}. Doing so yields an ad\`elic version of the parametrization in Theorem~\ref{thm:mainresult} that is useful to applying geometry-of-numbers methods over global fields. In particular, the ad\`elic viewpoint enables the computation of volumes of certain fundamental domains using Tamagawa numbers in \cite{HSAnalysis}.

	Let $K$ be a number field with ring of integers $R$.\footnote{We remark that analogues of the results and definitions in this section all hold for global fields $K$ with $\on{char}(K)\neq3$. For simplicity, we restrict to the case where $K$ is a number field. Converting the theory developed in this section to its analogue over global fields is simple; for an example of fully abstract, analogous statements, we refer the reader to Section~3 of \cite{BSW}.} Recall the following objects from \S\ref{subsec: param}. For a ring $A$, let $V(A)$ denote the vector space of triply symmetric binary cubic forms with coefficients in $A$, and let $V_\fa$ denote the subgroup of $V(K)$ of triply symmetric binary cubic forms of type $\fa$, i.e., elements of $V(K)$ of the form \[ax^3+3bx^2y+3cxy^2+dy^3\] for $a\in\fa$, $b\in R$, $c\in\fa\inv$, and $d\in\fa^{-2}$. There is a natural action of $\GL_2$ on $V$ given by \[g\cdot f(x,y)=\det(g)\inv f((x,y)g);\] the restriction of the action of $\GL_2(K)$ on $V(K)$ to the subgroup \[\GL(R\oplus\fa)=\left\{\bmat{r&x\\y&s}\mid r,s\in R,x\in\fa,y\in\fa\inv,rs-xy\in R^\times\right\}\] preserves $V_\fa$.

	\subsection{The action of $\GL_2$ on $V$ over local fields}\label{subsec: reduction theory}

	Given a place $\fp$ of $K$ and an element $x\in K$, let $|\cdot|_\fp$ denote the corresponding absolute value on $K$, and let $K_\fp$ denote the completion of $K$ with respect to $|\cdot|_\fp$. We begin by stating a well-known characterization of the action of $\GL_2(F)$ on $V(F)$ for an archimedean local field $F$.
	\begin{lemma}\label{lem: SL2 transitive action on V}
		The group $\GL_2(\C)$ acts transitively on the set of forms in $V(\C)$ with nonzero discriminant. The group $\GL_2(\R)$ acts transitively on the set of forms in $V(\R)$ with discriminant having a given sign.
	\end{lemma}
	
	Next, we explicitly construct a binary cubic form of each possible discriminant over $\R$ and $\C$. Let $F$ be an archimedean local field with absolute value $|\cdot|_F$. Given $\Delta\in F$, consider the form 
	\begin{equation}\label{eqn: form of fixed disc over R}
		f_{\Delta}(x,y)=|\Delta|_F^{1/4}\left(3x^2y+\frac{\arg(\Delta)}{4}y^3\right),
	\end{equation}
	where $\arg(\Delta)=\sign(\Delta)$ if $F=\R$ and where $\arg(\Delta)$ is the usual argument function if $F=\C$. We see immediately that $\disc(f_\Delta)=\Delta$.
	
	It is possible to compute the $\GL_2(F)$-stabilizers for the orbits of the $\GL_2(F)$-action on $V(F)$. Given a form $f\in V(F)$ with nonzero discriminant, the following well-known result tells us the size of $\Stab_{\GL_2(F)}(f)$ explicitly. We stress here that we are considering the discriminant of a form $f$ to be its \emph{reduced discriminant}, which is $-1/27$ times the usual discriminant used by Delone and Faddeev. Thus, the sizes of the stabilizers over $\R$ are opposite of the usual convention. 
	\begin{lemma}\label{lem: stab size}
		For $f\in V(\R)$, we have $\#\Stab_{\GL_2(\R)}(f)$ is 2 if $\disc(f)$ is positive and 6 if $\disc(f)$ is negative. For $f\in V(\C)$, we have $\#\Stab_{\GL_2(\C)}(f)=6$ if $\disc(f)\neq0$.
	\end{lemma}
	
	Later on, we will also need to understand the forms whose $x^3$-coefficient is 0. Let $P$ denote the subgroup of $\SL_2$ of lower-triangular matrices. Note that $P$ preserves the subspace $V^0$ of $V$ of forms whose $x^3$-coefficient is 0. For each $\Delta\in K_\fp$, let $V(K_\fp)^\Delta$ be the set of forms in $V(K_\fp)$ with discriminant $\Delta$.
	\begin{lemma}\label{lem: transitive action of parabolic subgroup}
		Let $\fp$ be any place of $K$. For $\Delta\in K_\fp^\times$, the group $P(K_\fp)$ of lower-triangular matrices in $\SL_2(K_\fp)$ acts transitively on the forms $V(K_\fp)^0$ whose discriminant is $\Delta$ and whose $x^3$-coefficient is 0.
	\end{lemma}
	\begin{proof}
		Let $f(x,y)=3bx^2y+3cxy^2+dy^3$ be a form in $V(K_\fp)^0$ with discriminant $\Delta\in K_\fp^\times$. First assume that $\fp\nmid 2$ so that 2 is a unit in $K_\fp$. Recall that $b^2|\disc(f)=\Delta$, so if $\Delta$ is a unit, then so is $b$. Replacing $f$ by some $P(K_\fp)$-translate, we may assume $b=1$. By translating $f$ by the element of $P(K_\fp)$ corresponding to the change of variables $(x,y)\mapsto(x-cy/2,y)$, we may assume $c=0$. Then $\disc(f)=-3c^2+4d=\Delta$ determines $d$, and we see that each of the sets of elements of $V(K_\fp)^\Delta$ with $x^3$-coefficient 0 form a single $P(K_\fp)$-orbit.
		
		When $\fp|2$, we may similarly replace $f$ with a $P$-translate such that $c\in\{0,1\}$. Then $\disc(f)=-3c^2+4d=\Delta$ determines $d$, and we see that each of the sets of elements of $V(K_\fp)^0$ with fixed discriminant form a single $P$-orbit.
	\end{proof}

	\subsection{Ad\`elic reformulation}\label{sec: volumes}
	We conclude by giving an ad\`elic reformulation of our setup and some additional machinery from the theory of reductive groups. For each fractional ideal $\fa$ of $R$, we define a congruence subgroup $\Gamma_\alpha$ of $\GL_2(\A)$ (here, $\A$ denotes the ring of ad\`eles of $K$). The congruence subgroup $\Gamma_\alpha$ acts naturally on a specific subgroup $\call_\fa$ of $V(\A)$, and we prove in Propositions~\ref{prop: Gamma is GL} and \ref{prop: L is V} that $\Gamma_\alpha\simeq\GL(R\oplus\fa)$ and that $\call_\alpha\simeq V_\alpha$, respectively. Immediately, we arrive at Theorem~\ref{thm: adelic param}---an ad\`elic version of Theorem~\ref{thm:mainresult} analogous to the ad\`elic parametrizations used by Bhargava, Shankar, and Wang in \cite{BSW}. Along the way, we prove Theorem~\ref{thm: adelic volumes}, which expresses $\GL_2(\A)/\GL_2(K)$ in terms of $\GL_2(K_\infty)/\GL(R\oplus\fa)$, where $K_\infty=\prod_{\fp|\infty}K_\fp$. The \emph{Tamagawa number} of $\GL_2$ is the volume of a certain subgroup of $\GL_2(\A)/\GL_2(K)$ with respect to a certain measure, and Theorem~\ref{thm: adelic volumes} allows us to write the sum of volumes of $\GL_2(K_\infty)/\GL(R\oplus\fa)$ in terms of the Tamagawa number of $\GL_2$, which is well-known to be 1.

	As usual, let $K$ be a global field with $\on{char}(K)\neq3$. Denote by $R$ the ring of integers of $K$. For each place $\nu$ of $K$, let $K_\nu$ denote the completion of $K$ at $\nu$, and let $R_\nu$ denote the ring of integers of $K_\nu$ when $\nu$ is nonarchimedean and $K_\nu$ itself when $\nu$ is archimedean. Let $\A$ denote the ring of ad\`eles of $K$, and let $\A^\times$ denote its unit group of id\`eles. For each finite set of places $S$, we have the ring of $S$-\emph{ad\`{e}les} \[\A^S=\prod_{\nu\in S}K_\nu\times\prod_{\nu\not\in S}R_\nu\] and its subgroup of $S$-\emph{id\`{e}les}, defined to be \[I_K^S=\prod_{\nu\not\in S}U_\nu\times \prod_{\nu\not\in S}K_\nu^\times,\] where $U_\nu$ is $\C^\times$ for $\nu$ infinite complex and $\R_{+}$ for $\nu$ infinite real. Complementary to the ring of $S$-ad\`{e}les is the ring $K_S$, which is the product of the completions of $K$ at the places in $S$. Both $\A^S$ and $K_S$ embed into $\A$ by setting all other coordinates to 1. When $S$ is the set of infinite places, we refer to the $S$-ad\`{e}les and $S$-id\`{e}les as $\infty$-ad\`{e}les and $\infty$-id\`{e}les, respectively. In this case, we denote $K_S$ by $K_\infty$, and note that $K_\infty$ is isomorphic as a real vector space to $\R^{[K:\Q]}$. We may decompose $\A$ in terms of these subrings as $\A=K_\infty\times (K\otimes\widehat{R}),$ where $\widehat{R}=\prod_{\fp\nmid\infty}R_\fp$.
	
	The subgroup of principal ad\`eles are the elements of $\A$ in the image of the diagonal embedding $K\hookrightarrow\A$, and under this embedding, any fractional ideal $\fa$ of $K$ is situated as a lattice in $K_\infty$ with covolume $N(\fa)\Delta_K^{1/2}$. For a finite set of places $S$, the intersection of the principal id\`{e}les with the $S$-id\`{e}les $K\cap I_K^S$ is called the group of $S$-\emph{units} and is denoted $K^S$. If $S$ is the set of infinite places, then $K^S$ is the group of units $R^\times$ in $R$.

	Let $|\cdot|\colon\A\to\R_{\geq0}$ denote the usual norm map on $\A$, which is a continuous homomorphism given explicitly by piecing together the absolute values $|\cdot|_\fp$ on each $K_\fp$: \begin{equation}\label{eqn: adele norm}|(a_\fp)_\fp|=\prod_\fp|a_\fp|_\fp.\end{equation} Via restriction, the ad\`ele norm (alternatively, the {ad\`elic absolute value}) gives norms on the various subrings and subgroups of $\A$ defined in the above. Notably, for $a$ an id\`ele, the product on the right-hand side of \eqref{eqn: adele norm} is finite, and for $a\in\A\setminus\A^\times$, we have $|a|=0$. Finally, recall the product formula: for all $a\in K^\times$, we have \begin{equation}\label{eqn: product formula}
		|a|=\prod_{\fp}|a|_\fp=1.
	\end{equation} Given any subset $U$ of $\A$, we take $U^1$ to be the elements of $U$ with ad\`ele norm 1, i.e., the intersection of $U$ with the kernel of the norm map.
	
	Let $C_K=\A^\times/K^\times$ denote the id\`ele class group of $K$. Recall that there is a surjective homomorphism from $\A^\times\to\cali(K)$ taking an id\`ele $a=(a_\fp)_\fp\in\A^\times$ to the fractional ideal of $K$ given by \[a\mapsto \prod_{\fp\nmid\infty}\fp^{\nu_\fp(a_\fp)}.\] By unique factorization, this map descends to a surjective homomorphism from $C_K\to\Cl(K)$ whose kernel is $(\A^\infty)^\times$. Thus, we may realize the class group of $K$ as the double coset space \begin{equation}\label{eqn: class group GL1}
		\Cl(K)\simeq(\A^\infty)^\times\backslash\A^\times/K^\times.
	\end{equation}

	Regarding $\A^\times$, $(\A^\infty)^\times$, and $K^\times$ as $\GL_1(\A^\times)$, $\GL_1((\A^\infty)^\times)$, and $\GL_1(K)$, respectively, we define the \emph{class group} of $\GL_k(\A)$ to be the double coset space \[\Cl(\GL_k(\A))=\GL_k(\A^\infty)\backslash\GL_k(\A)/\GL_k(K).\] In analogy with \eqref{eqn: class group GL1}, we have the following theorem.
	\begin{theorem}[cf.\ {\cite[Proposition~8.1]{PR}}]\label{thm: class group GL}
		Let $K$ be a global field. Then $\Cl(\GL_k(\A))\simeq\Cl(K)$.
	\end{theorem}

	Choose a set of representatives for the elements of $\Cl(K)$. Henceforth, when we reference an ideal class $[\fa]$ of $\Cl(K)$, we are referring to this specific choice of ideal representative $\fa$. Leveraging the surjectivity of the map $C_K\to\Cl(K)$, we lift our ideal class representatives to a set of id\`ele class representatives. In other words, for each ideal class representative $\fa$, let $\alpha$ be an id\`ele whose corresponding ideal is $\fa$ and whose entries at the archimedean places are all 1. In the sequel, we reserve the notation $\fa$ and $\alpha$ for these ideal and id\`ele representatives, exclusively. 
	
	Using the correspondence laid out in Theorem~\ref{thm: class group GL}, we construct a set of representatives for the classes of $\GL_2(\A)$ by constructing an element of determinant $\alpha$ for each ideal class representative $\fa$ of $\Cl(K)$. Consider the element \[\bm{\alpha}=\bmat{&\alpha\\-1&}\in\GL_2(\A).\] Theorem~\ref{thm: class group GL} tells us that \begin{equation}\label{eqn: class gp decomp}
		\GL_2(\A)=\bigsqcup_{\alpha\in\Cl(K)}\GL_2(\A^\infty)\bm{\alpha}\GL_2(K),
	\end{equation} where the disjoint union runs over the chosen id\`ele representatives of the class group of $K$.

	Using $\bm{\alpha}$, we may write $\GL(R\oplus\fa)$ ad\`elically. For each id\`ele representative $\alpha$ of $\Cl(K)$, we define the ${\alpha}$\emph{-congruence subgroup}  \begin{equation}\label{eqn: Gamma def}\Gamma_\alpha=\GL_2(K)\cap\bm{\alpha}\inv\GL_2(\A^\infty)\bm{\alpha}=\GL_2(K)\cap\GL_2(K_\infty)\bm{\alpha}\inv\GL_2(\widehat{R})\bm{\alpha},
	\end{equation} where the second equality follows from our choice of $\alpha$ and $\bm{\alpha}$. Combining \eqref{eqn: class gp decomp} and \eqref{eqn: Gamma def}, we can express $\GL_2(\A)/\GL_2(K)$ in terms of $\GL_2(K_\infty)/\Gamma_\alpha$.  
	\begin{theorem}\label{thm: adelic volumes}
		For a number field $K$, we have that \[\GL_2(\A)/\GL_2(K)=\bigsqcup_{\alpha\in\Cl(K)}\GL_2(\widehat{R})\times\GL_2(K_\infty)/\Gamma_\alpha.\] 
	\end{theorem}
	\begin{proof}
		Note that it suffices to show that \[\GL_2(\widehat{R})\backslash\GL_2(\A)/\GL_2(K)=\bigsqcup_{\alpha\in\Cl(K)}\GL_2(K_\infty)/\Gamma_\alpha.\] Define a map $\phi\colon \bigsqcup_\alpha\GL_2(K_\infty)/\Gamma_\alpha\to \GL_2(\widehat{R})\backslash\GL_2(\A)/\GL_2(K)$ that, for each $\alpha\in\Cl(K)$, sends $g\Gamma_\alpha\mapsto \GL_2(\widehat{R})\bm{\alpha}g\GL_2(K)$. The well-definedness of $\phi$ follows immediately from the fact that $\Gamma_\alpha\subset\GL_2(K)$. 
		
		Now, by Theorem~\ref{thm: class group GL} and our choice of $\bm{\alpha}$, we may write \[\GL_2(\A)=\bigsqcup_{\alpha\in\Cl(K)}\GL_2(\widehat{R})\GL_2(K_\infty)\bm{\alpha}\GL_2(K)=\bigsqcup_{\alpha\in\Cl(K)}\GL_2(\widehat{R})\bm{\alpha}\GL_2(K_\infty)\GL_2(K),\] which tells us that our map is surjective.
		
		For injectivity, suppose that $\GL_2(\widehat{R})\bm{\alpha}g\GL_2(K)=\GL_2(\widehat{R})\bm{\alpha}h\GL_2(K)$ for $g,h\in\GL_2(K_\infty)$. Then there exist $x\in\GL_2(\widehat{R})$ and $y\in\GL_2(K)$ such that $\bm{\alpha}g=x\bm{\alpha}hy.$ Write $ gy\inv h\inv=\bm{\alpha}\inv x\bm{\alpha}$. It follows that \[y=h\inv\bm{\alpha}\inv x\inv\bm{\alpha}g=h\inv g\bm{\alpha}\inv x\inv\bm{\alpha},\]  forcing $y\in\Gamma_\alpha$.  Considering that $g$, $y\inv$, and $h$ are all 1 in each nonarchimedean coordinate and that $x$ and $\bm{\alpha}$ are all 1 in each archimedean coordinate, we must have $g=hy$, as desired.
	\end{proof}
	
	Regarded as a subgroup of $\GL_2(K)$, it turns out that $\Gamma_\alpha$ is a familiar object. 
	\begin{prop}\label{prop: Gamma is GL}
		As a subgroup of $\GL_2(K)$, the $\alpha$-congruence subgroup $\Gamma_\alpha$ is isomorphic to $\GL(R\oplus\fa)$.
	\end{prop}
	\begin{proof}
		We begin by showing that the diagonal image of $\GL(R\oplus\fa)$ is contained in $\Gamma_\alpha$. Let \[\gamma=\bmat{a&b\\c&d}\] be an element of $\GL(R\oplus\fa)$, where $a,d\in R$, $b\in\fa$, and $c\in\fa\inv$. Matrix multiplication tells us that $\bm{\alpha}\gamma\bm{\alpha}\inv$ is contained in $\GL_2(\A^\infty)$, since $v_\fp(\alpha c)\geq0$ and $v_\fp(\alpha\inv b)\geq0$ for all finite $\fp$. 
		
		To see surjectivity, we begin by noting that any element $g$ of $\Gamma_\alpha$ with entries in $K$ can be written as $g=\bm{\alpha\inv}h\bm{\alpha}$ for $h\in \GL_2(\A^\infty)$. Supposing \[g=\bmat{p&q\\r&s}\] for $p,q,r,s\in K$, a similar matrix computation shows that, for each $\fp|\fa$, we have $v_\fp(q)\geq v_\fp(\fa)$ and $v_\fp(r)\leq v_\fp(\fa\inv)$. It follows that $q\in\fa$ and $r\in\fa\inv$, as desired.
	\end{proof} 
	
	Upon realizing that $\Gamma_\alpha$ and $\GL(R\oplus\fa)$ are isomorphic, it is natural to ask whether there exists an ad\`elic analogue of the lattice $V_\fa$. From the $\GL_2$-action on $V$, we inherit a natural action of $\Gamma_\alpha$ on the lattice \begin{equation}\label{eqn: call def}\call_{\alpha}=V(K)\cap\bm{\alpha}\inv V(\A^\infty)=V(K)\cap V(K_\infty)\times\bm{\alpha}\inv \prod_{\fp\nmid\infty}V(R_\fp).
	\end{equation} In analogy with Proposition~\ref{prop: Gamma is GL}, we have the following result.
	\begin{prop}\label{prop: L is V}
		As a subgroup of $V(K)$, we have that $\call_\alpha$ is isomorphic to $V_\fa$.
	\end{prop}
	The proof of Proposition~\ref{prop: L is V} is nearly identical to that of Proposition~\ref{prop: Gamma is GL} and is therefore omitted.  Propositions~\ref{prop: Gamma is GL} and \ref{prop: L is V} then imply the following ad\`elic version of Theorem~\ref{thm:mainresult}.
	\begin{theorem}[cf.\ {\cite[Theorem~11]{BSW}}]\label{thm: adelic param}
		With notation as in the above, for each $\fa$, we have a bijection \[\{\fa\textrm{-balanced quadruples }(S,I,\delta,s)\}/\approx\,\longleftrightarrow\Gamma_\alpha\backslash\,\call_\alpha.\] Furthermore, if $\ell$ is an element of $\call_\alpha$ corresponding to some balanced quadruple $(S,I,\delta,s)$, we have that \[\Disc_{S/R}=\fa^2\disc(\ell).\]
	\end{theorem}

	\section*{Acknowledgments}
	
	\noindent This paper is based in part on the first-named author’s senior thesis at Harvard College. We are very grateful to Melanie Matchett Wood for providing invaluable advice and encouragement throughout our research. We also thank Barry Mazur, Arul Shankar, and Artane Siad for helpful conversations and suggestions. The first-named author was supported by the NSF under Award No.~DMS-2140043. The second-named author was supported by the NSF under Award No.~2202839.
	
	\bibliographystyle{alpha}
	\bibliography{bib}
	
\end{document}